\documentclass[11pt]{article}
\usepackage[dvips]{graphicx}
\usepackage{psfig,subfigure,verbatim,indentfirst,cleveref}
\usepackage{amsmath,amssymb,amsthm}

\usepackage{amsmath}
\usepackage{cite}
\usepackage{graphicx}
 \usepackage{color}
 \newtheorem{theorem}{Theorem}[section]
\newtheorem{lemma}[theorem]{Lemma}
\newtheorem{pro}[theorem]{Proposition}

\theoremstyle{definition}
\newtheorem{definition}[theorem]{Definition}

\theoremstyle{remark}
\newtheorem{remark}[theorem]{Remark}

\numberwithin{equation}{section}

\def\be {\begin{equation}}
\def\ee {\end{equation}}
\def\ba {\begin{eqnarray}}
\def\ea {\end{eqnarray}}
\newcommand{\f}{\frac}

\newcommand{\R}{{\mathbb R}}
\newcommand{\N}{{\mathbb N}}
\newcommand{\ds}{\displaystyle}

\begin{document}
\title{\Large\bf{Formulation of the normal forms of Turing-Hopf bifurcation in reaction-diffusion systems with time delay\thanks{Supported by the National Natural Science Foundation of China (No.11371112), US-NSF grant DMS-1715651.}}}
\author{{Weihua Jiang\thanks {Corresponding author. E-mail address: jiangwh@hit.edu.cn}, \; Qi An}\\\footnotesize {\em Department of Mathematics, Harbin Institute of Technology, Harbin 150001, P.R.China }\\
{Junping Shi}\\\footnotesize {\em Department of Mathematics, College of William and Mary, Williamsburg, Virginia, 23187-8795, USA}
            }

\date{}
	\maketitle
	\baselineskip=0.9\normalbaselineskip \vspace{-3pt}
	
\begin{abstract}
The normal forms up to the third order for a Hopf-steady state bifurcation of a general system of partial functional differential equations (PFDEs) is derived based on the center manifold and normal form theory of PFDEs. This is a codimension-two degenerate bifurcation with the characteristic equation having a pair of simple purely imaginary roots and a simple zero root, and the corresponding eigenfunctions may be spatially inhomogeneous. The PFDEs are reduced to a three-dimensional system of ordinary differential equations and precise dynamics near bifurcation point can be revealed by  two unfolding parameters. The normal forms are explicitly written as functions of the Fr\'echet derivatives up to the third orders and characteristic functions of the original PFDEs,  and they are presented in a concise matrix notation, which greatly eases the applications to the original PFDEs and  is convenient for computer implementation. This provides a user-friendly approach of showing the existence and stability of patterned stationary and time-periodic solutions with spatial heterogeneity when the parameters are near a Turing-Hopf bifurcation point, and it can also be applied to reaction-diffusion systems without delay and the retarded functional differential equations without diffusion.
		
\noindent
{\small {\bf Keywords:} Reaction-diffusion equations with time-delay; Hopf-steady state bifurcation; Turing-Hopf bifurcation; Normal form}
\end{abstract}
	
\section{Introduction}

In a dynamical mathematical model, the asymptotic behavior of the system often changes when some parameter  moves across certain threshold values and such phenomenon is called a bifurcation. The method of the normal forms is a standard and effective tool to analyze and simplify  bifurcation problems, see \cite{Chow1982,Guck1983,Wigg1990}.  The main idea is to transform the differential equations to a topologically conjugate normal form near the singularity. For ordinary differential equations (ODEs), the methods of computing the normal forms have been developed in, for example, \cite{Brjuno,Chow1982,Chow1994,Golubitsky1985,Guck1983},  and for functional differential equations (FDEs), similar methods have also been developed in, for example, \cite{Hale1977,Faria1,Faria2}. In FDEs, due to the effect of time-delays, Hopf bifurcations occur more frequently which destabilize a stable equilibrium and produce temporal oscillatory patterns \cite{Hale1977}.  Often a center manifold reduction reduces a higher dimensional problem to a lower dimensional one, and the normal form can be computed on the center manifold \cite{Carr1981,Hassard1981,Marsden1976,Vanderbauwhede1989}. %In \cite{Faria1, Faria2}, a novel method of computing normal forms for retarded functional differential equations
%(PFDEs), in which the coefficients of normal forms are explicitly
%expressed in terms of parameters of the original equations, was developed.
The method  proposed in \cite{Faria1, Faria2} has been applied to the bifurcation problems in ODEs or FDEs, such as codimension-one Hopf bifurcation, and codimension-two Hopf-zero bifurcation and Bogdanov-Takens bifurcation etc., see    \cite{XiaoR2001,JiangYuan2007,CampYuan2008,JiangW2010,WangJ2010,YuanRJiang2015,NiuBJiang2013,WangC2008}.%SongYZou2014,SongY2016}.

The methods of center manifolds and normal forms have also been extended to many partial differential equations (PDEs). For example, the existence of center manifolds or other invariant manifolds for semilinear parabolic equations have been proved in, for example,  \cite{BJ1989,BLZ1998,ChowLu1988,Henry}. For parabolic equations with time-delay or functional partial differential equations (FPDEs), the existence and smoothness of center manifolds have also been established in \cite{Faria2000,Faria2001,Faria2002,Lin1992,Wu2004,Wu1996,Wu1998}.  In particular, the calculation of  normal form on the centre manifold of FPDEs was provided in \cite{Faria2000,Faria2001}. These theories can be applied to reaction-diffusion systems (with or without time-delays) which appear in many applications from physics, chemistry and biology. For example, the existence of Hopf bifurcations and associated stability switches have been considered in many recent work \cite{Guo2016,Hadeler2007,HuangW1996,Yi2009,Su2009,Yi2010,Su2010,YuanXP2010,Zuo2011,ChenS2012,YiFQ2013,Guo2015}. More recently with the integrated semigroup theory,  the center manifold and normal form theory for semilinear equations with non-dense domain have  also been developed \cite{Ruan2009,Liu2011,Liu2014,Liu2016}.

An important application of normal form theory for PDEs and FPDEs is the formation and bifurcation of spatiotemporal patterns in reaction-diffusion systems (with delays) from various physical, chemical and biological models. In the pioneer work or Turing \cite{Turing1952}, it was shown that  diffusion could destabilize an otherwise stable spatially homogeneous equilibrium of a reaction-diffusion system, which leads to the spontaneous formation of spatially inhomogeneous pattern. This phenomenon is often called the Turing instability or diffusion-driven instability, and associated Turing bifurcation could lead to spatially inhomogeneous steady states \cite{kondo2010reaction,Maini2017,Shi2009}. Such Turing type pattern formation mechanisms have been verified in several recent chemical or biological studies \cite{KondoAsai1995,lengyel1991modeling,muller2012differential,sheth2012hox}.

In many reaction-diffusion models, temporal oscillation caused by Hopf bifurcation and  spatial patterns from Turing mechanism can occur simultaneously to produce Turing-Hopf patterns which oscillate in both space and time \cite{Baurmann2007,DeWit1997,Ricard2009,Rovinsky1992,Maini1997}. Mathematically the  complex spatiotemporal Turing-Hopf patterns involves the interaction of the dynamical properties of  two Fourier modes, and it can be analyzed through unfolding a codimension-two Turing-Hopf bifurcation, see \cite{Holmes1997,SongYZou2014,Ruan2015,SongY2016} and references therein. It is also a Hopf-zero bifurcation with the zero eigenvalue corresponding to a spatially inhomogeneous eigenfunction.

The aim of the present paper is  to provide  the computation of the normal form up to the third order at a known steady state solution for a reaction-diffusion system with time-delay. This normal form can be used to unfold the complex spatiotemporal dynamics near a  Turing-Hopf bifurcation point with one of the unfolding parameter being the time-delay. We follow the framework of Faria \cite{Faria2000,Faria2002} to reduce the general PFDEs with perturbation parameters to a three-dimensional systems of ODEs up to third order, restricted on the local center manifold near a Hopf-steady state type of singularity, and the unfolding parameters can be expressed by those original perturbation parameters. Usually the third order normal form is sufficient for analyzing the bifurcation phenomena in most of the applications. The reduced three-dimensional ODE system  can be further transformed to a two-dimensional amplitude system and the bifurcation analysis can be carried out following \cite{Guck1983} to provide precise dynamical behavior of the system using the two-dimensional unfolding parameters. Furthermore we give an explicit formula of the coefficients in the truncated normal form up to third order for the Hopf-steady state bifurcation of delayed reaction-diffusion equations with Neumann boundary condition, which includes the important application to the Turing-Hopf bifurcation.

Our approach in this paper has several new features compared to previous work. First our basic setup of PFDE systems  follows the assumptions ($H1$)-($H4$) in \cite{Faria2000} with slight changes to fit our situation, but we remove the more restrictive assumption ($H5$) which was used in  \cite{Faria2000}. Hence our computation of normal forms can be applied to more general situations. Secondly the normal form formulas here are directly expressed  by the Fr\'echet derivatives up to the third orders and characteristic functions of the original PFDEs, not the reduced ODEs. Hence one can apply our results directly to the original PFDEs without the reduction steps. Also our formulas of the normal form are presented in a concise matrix notation which also eases the applications. Thirdly we neglect the higher order ($\geq 2$) dependence of the perturbation parameters on the system, as in practical application, the influence of the small perturbation parameter on the dynamics of the system is mainly  linear. This again simplifies the normal form but still fulfills the need in most applications. Finally we remark that the normal form formulas developed in this paper for PFDEs are also applicable to the general reaction-diffusion equations without delay (PDEs) and  the delay differential equations without diffusion (FDEs) with obvious adaption, and the unfolding parameters are not necessarily the time-delay or diffusion coefficients.

Because the coefficients of  computed normal form can be explicitly expressed using the information from the original system, our algorithm enables us to draw conclusions on the impact of original system parameters on the dynamical behavior near the Turing-Hopf singularity. To illustrate our normal form computation and application algorithm, we apply our methods to the Turing-Hopf bifurcation in a diffusive Schnakenberg type chemical reaction system with gene expression time delay proposed in \cite{SER}. Turing and Hopf bifurcations for this system have been considered in \cite{Chen2013JNS,YGLM}, and Turing-Hopf bifurcation for the system in a different set of parameters was recently considered in \cite{JWC2017}.

The rest of the paper is organized as follows. In Section 2, the framework of  the system of PFDEs at a Hopf-steady state singularity and the  phase space  decomposition are given,  and the reduction of the original equations to a three-dimensional ODE system is introduced. The formulas of normal form up to third order are presented in Section 3 while the proof is postponed to Section \ref{proof}. In Section 4, the precise formulas of the normal forms with the Neumann boundary condition and the spatial domain $\Omega=(0,l\pi)$ are given. The application of abstract formulas to the example of diffusive Schnakenberg system with gene expression time delay  is shown in Section 5, and some concluding remarks are given in Section 7.

\section{Reduction based on phase space decomposition}
%	\section{Reduction and  the normal form for PFDEs with a Hopf-steady state singular}
In this section, we discuss the reduction and the normal forms for a system of PFDEs subject to homogeneous Neumann or Dirichlet boundary conditions at a  Hopf-steady state singularity  with original perturbation parameters following the methods in \cite{Faria1,Faria2}. We will show that the system of PFDEs  can be reduced to a three-dimensional system of  ordinary differential equations defined on its center manifold.

Assume that $\Omega$ is a bounded open subset of $\mathbb{R}^n$ with smooth boundary, and $X$ is a Hilbert space of complex-valued functions defined on $\bar{\Omega}$ with inner product $\langle\cdot,\cdot\rangle$. Denote by $\mathbb{N}_B$  the set of nonnegative integers or positive integers,  depending on the boundary condition:
$$
\mathbb{N}_B=\left\{
\begin{array}{ll}
\mathbb{N}\cup\{0\},& \mathrm{for~ homogeneous~ Neumann~
	boundary ~conditions},\\
\mathbb{N},& \mathrm{for~ homogeneous~ Dirichlet~
	boundary ~conditions}.
\end{array}
\right.
$$
Let $\{\mu_k:k\in \mathbb{N}_B \}$ be the set of eigenvalues of $-\Delta$ on $\Omega$ subject to homogeneous Neumann  or Dirichlet boundary conditions, that is
$$\Delta\beta_k+\mu_k\beta_k=0,\;\;x\in \Omega, \;\; \frac{\partial u}{\partial n}=0 \; \text{ or } \; u=0, \;\; x\in \partial\Omega.$$
Then we have
\begin{equation*}
  \begin{split}
     &0=\mu_0<\mu_1\leq \cdots\leq\mu_k\leq \cdots\rightarrow \infty,~\mathrm{for~Neumann
	~boundary~conditions, or} \\
      &0<\mu_1<\mu_2\leq \cdots\leq\mu_k\leq \cdots\rightarrow \infty,~\mathrm{for~Dirichlet
	~boundary~conditions},
  \end{split}
\end{equation*}
and the corresponding eigenfunctions $\{\beta_k:k\in \mathbb{N}_B\}$  form an orthonormal basis of $X$. Fixing $m\in {\mathbb N}$ and $r>0$, define $\mathcal{C}=C([-r,0];X^m)\ (r>0)$ to be the Banach space of
	continuous maps from $[-r,0]$ to $X^m$ with the sup norm.
	%We write $u_t \in \mathcal C$ for $u_t(\theta) = u(t + \theta), -r \leq\theta \leq 0 $.
	%We consider PFDEs with an equilibrium point at the origin, given in abstract form as
	%Based on the framework of \cite{Faria2000},
We consider an abstract PFDE with parameters in the phase space $\mathcal{C}$ defined as
\begin{equation}\label{eq301}
\dot{u}(t) = D(\alpha)\Delta u(t)+L(\alpha) u_t + G(u_t,\alpha),
\end{equation}
where $u_t\in \mathcal{C}$ is defined by $u_t(\theta)=u(t+\theta)$ for $-r\leq\theta\leq 0$, $D(\alpha)=\mathrm{diag}(d_1(\alpha),d_2(\alpha),\ldots,d_m(\alpha))$ with  $d_i(0)>0$ for $1\le i\le m$; the domain of $\Delta u(t)$ is defined by $dom(\Delta)=Y^m \subseteq X^m$ where $Y$ is defined as
 \begin{equation*}
   \begin{split}
      &Y=\ds\left\{u\in W^{2,2}(\Omega): \frac{\partial u}{\partial n}=0,\; x\in \partial\Omega\right\}~\mathrm{for~Neumann
	~boundary~conditions, or} \\
       &Y=\ds\left\{u\in W^{2,2}(\Omega): u=0,\; x\in \partial\Omega\right\}~\mathrm{for~Dirichlet
	~boundary~conditions};
   \end{split}
 \end{equation*}
the parameter vector $~\alpha=(\alpha_1, \alpha_2)$ is in a neighborhood $V\subset \R^2$ of $(0,0)$,
$L:~V\rightarrow L(\mathcal{C},\R^m)$ (the set of linear mappings) is $C^1$ smooth, $G: \mathcal{C}\times V\rightarrow
	\R^m$ is $C^k$ smooth for $k\geq 3$,  $G(0,0)=0$, and the Jacobian matrix $D_{\varphi}G(0,0)=0$ with $\varphi\in \mathcal{C}$.
	
%	Obviously, \eqref{eq301} has a trivial equilibrium.
Let $L_0=L(0)$ and $ D_0=D(0)$. Then the linearized equation about the zero equilibrium  of  \eqref{eq301} can be written as
\begin{equation}\label{eq302}
\dot{u}(t) = D_0\Delta u(t)+L_0 u_t.
\end{equation}
We impose the following hypotheses (similar to \cite{Faria2000}):

\begin{enumerate}
  \item[{\bf (H1)}] $D_0\Delta$ generates a $C_0$ semigroup $\{T(t)\}_{t\geq 0}$ on $X^m$ with $|T(t)| \leq Me^{\omega t}$ for all $t \geq 0$, where  $M \geq 1$, $\omega \in \mathbb{R}$, and $T(t)$ is a compact operator for each $t > 0$;
  \item[{\bf (H2)}] $L_0$ can be extended to a bounded linear operator from $B\mathcal{C}$ to $X^m$, where
	$B\mathcal{C} = \{\psi \in {\mathcal C}: \ds\lim_{\theta \rightarrow 0^-} \psi(\theta) \text{ exists} \}$ with the
	sup norm.
  \item[{\bf (H3)}] the subspaces $\mathcal{B}_k=\{\langle v(\cdot),\beta_k \rangle \beta_k: v\in \mathcal{C}\}\subset\mathcal{C}$ ($k\in \N_B$) satisfy
$L_0(\mathcal{B}_k)\subseteq span \{e_i\beta_k:1\le i\le m\}$,
where $\{e_i: 1\le i\le m\}$ is the canonical basis of $\R^m$, and
$$\langle v,\beta_k \rangle=(\langle v_1,\beta_k \rangle,~\langle v_2,\beta_k \rangle,~\ldots,~\langle v_m, \beta_k \rangle)^{\mathrm{T}},~k\in \mathbb{N}_B,
~\mathrm{for}~v=(v_1,v_2,\ldots,v_m)^{\mathrm{T}}\in\mathcal{C}.$$
\end{enumerate}
	
Let $A$ be the infinitesimal generator associated with the semiflow of the linearized equation \eqref{eq302}. It is known that $A$ is given by
$$(A\phi)(\theta)=\dot{\phi}(\theta),~dom(A)=\{\phi\in\mathcal{C}:\dot{\phi}\in\mathcal{C},\phi(0)\in dom(\Delta),\dot{\phi}(0)=D_0\Delta\phi(0)+L_0\phi\},$$
and the spectrum $\sigma(A)$ of $A$ coincides with its point spectrum $\sigma_P(A)$. Moreover $\lambda\in\mathbb{C}$ is in $\sigma_P(A)$ if and only if there exists $y\in dom(\Delta)\setminus \{0\}$ such that $\lambda$ satisfies
\begin{equation}\label{eq304}
\triangle(\lambda)y = 0, \ \ \mbox{with}\ \ \triangle(\lambda)=\lambda
I-D_0\Delta-L_0(e^{\lambda \cdot}I).
\end{equation}
	
	%Since
%	%{\tcr{
%    $$
%	L_0(\langle v(\cdot),\beta_k \rangle \beta_k)=L_0(\langle v(\cdot),\beta_k\rangle)\beta_k,~~~~~~~~v\in\mathcal{C},~k\in \mathbb{N}_B,
%	$$
%    %}}
%	the subspaces $\mathcal{B}_k=\{\langle v(\cdot),\beta_k \rangle \beta_k|v\in \mathcal{C}\}\subset\mathcal{C}$ satisfy
%	\begin{equation}\label{eq304+}
%	L_0(\mathcal{B}_k)\subseteq span \{e_i\beta_k,i=1,2,\ldots,m\}\,,
%	\end{equation}
%	where $e_1,e_2,\ldots,e_m$ are the natural basis of $\mathbb{R}^m$, and
%	$$\langle v,\beta_k \rangle=(\langle v_1,\beta_k \rangle,~\langle v_2,\beta_k \rangle,~\ldots,~\langle v_m, \beta_k \rangle)^{\mathrm{T}},~k\in \mathbb{N}_B,~\mathrm{for}~v=(v_1,v_2,\ldots,v_m)^{\mathrm{T}}\in\mathcal{C}.$$
%    Thus the hypotheses (H3) of \cite{Faria2000} is naturally satisfied, we have no need of any additional hypotheses.
	
	%Moreover,  can be considered a subspace of $\mathcal{C}$, thus $L_0$ can be to act on $C$.
	
By using the decomposition of $X$ by $\{\beta_k\}_{k\in \mathbb{N}_B}$ and $\mathcal{B}_k$, the equation $\triangle(\lambda)y = 0$, for some $y\in dom(\Delta)\setminus \{0\}$, is equivalent to
a sequence of characteristic equations
\begin{equation}\label{eq304+1}
\mathrm{det}\triangle_k(\lambda)=0, \ \ \mbox{with}\ \
\triangle_k(\lambda)=\lambda I-\mu_k D_0-L_0(e^{\lambda \cdot}I), \;\; k\in \mathbb{N}_B.
\end{equation}
Here $L_0:C \rightarrow \mathbb{C}^m $ where $C\triangleq C([-r,0];\mathbb {C}^m)$. Then for any $k\in \mathbb{N}_B$, on $\mathcal{B}_k$, the linear equation \eqref{eq302} is equivalent to a Functional Differential Equation (FDE)
on $\mathbb{C}^m$:
\begin{equation}\label{eq302+}
\dot{z}(t) = -\mu_k D_0z(t)+L_0 z_t,
\end{equation}
with characteristic equation \eqref{eq304+1},  where $z_t(\cdot)=\langle u_t(\cdot),\beta_k\rangle \in C$. For any $k\in \mathbb{N}_B$, we also denote by $\eta_k\in BV([-r,0],\mathbb{C}^m)$ to be the $m\times m$ matrix-valued function of bounded
variation defined on $[-r,0]$ such that
\begin{equation}\label{eq303}
-\mu_kD_0 \psi(0)+L_0\psi = \int_{-r}^0 \mathrm{d}\eta_k(\theta)\psi(\theta),\;\; ~\psi\in C.
\end{equation}
The adjoint bilinear form on $C^\ast\times C$, where $C^\ast \triangleq C([0,r];~ \mathbb{C}^{m\ast})$, is defined by
\begin{equation}\label{eq5}
(\psi,\ \varphi)_{k}=\psi(0)\varphi(0)-\int^0_{-r}\int^\theta_0\psi(\xi-\theta)\mathrm{d}\eta_{k}(\theta)\varphi(\xi)\mathrm{d}\xi,\ \ \psi\in C^\ast,\ \varphi\in C.
	\end{equation}

We make the following basic assumption on a Hopf-steady state bifurcation point:
\begin{enumerate}
\item[{\bf (H4)}] There exists a neighborhood $V\subset {\mathbb R}^2$ of zero such that for $\alpha := (\alpha_1,\alpha_2) \in V$, the characteristic equation \eqref{eq304+1}  with $k=k_1\in \N_B$ has a simple real
    eigenvalue $\gamma(\alpha)$ with  $\gamma(0)=0$, $\ds\frac{\partial\gamma}{\partial \alpha_2}(0)\neq 0$,
    and \eqref{eq304+1} with $k=k_2\in \N_B$ has a pair of simple complex  conjugate eigenvalues $\nu(\alpha)\pm \mathrm{i}\omega (\alpha)$ with $\nu(0)=0$, $\omega(0)=\omega_0>0$, $\ds\frac{\partial \nu}{\partial \alpha_1}(0)\neq 0$, all other eigenvalues of \eqref{eq304} have non-zero real part for $\alpha \in V$.
\end{enumerate}

\begin{definition}\label{def2.1}
We say that a $(k_1,k_2)-$mode  Hopf-steady state bifurcation occurs for \eqref{eq301} near the trivial equilibrium at $\alpha=(0,0)$ if assumptions {\bf (H1)}-{\bf (H4)} are satisfied, or briefly,  a Hopf-steady state bifurcation occurs. Moreover, if $k_1\neq 0$, we say that a $(k_1,k_2)-$mode  Turing-Hopf bifurcation occurs, or briefly, a Turing-Hopf bifurcation occurs.
\end{definition}

%	In the following, we assume that $\Lambda=\{\lambda \in \mathbb{C},~\lambda$ is a solution of the equations (\ref{eq304+1}) for $\alpha=0$ with $k=k_1,~k_2$ and  $\mathrm{Re}\lambda=0 \}$. Obviously, $\Lambda=\{0,\pm i \omega_0\}$.
Let $\Lambda_1=\{0\},~\Lambda_2=\{\pm \mathrm{i} \omega_0\},$ and $\Lambda=\Lambda_1\cup\Lambda_2$.  Then the phase space $C$ is decomposed by $\Lambda_i$:
	%$\Lambda_i \triangleq \{\lambda \in \mathbb{C} :\lambda $ satisfies (\ref{eq304+1}) for $\alpha=0$ with $k=k_i$ and $\mathrm{Re}\lambda=0\}$ as
$$C = P_i\oplus Q_i,$$
where
$Q_i = \{\varphi \in C : (\psi, \varphi )_{k_i}= 0,~\text{for all}~\psi \in P_i^*\},~i=1,2$.
We choose the basis
\begin{equation}\label{eq5+3}
\Phi_1=\phi_1,~\Psi_1=\psi_1,~\Phi_2=(\phi_2,\bar{\phi_2}),~\Psi_2=\left(
\begin{array}{c}
\psi_2\\
\bar{\psi_2}
\end{array}\right)
\end{equation}
in $P_1,~P_1^\ast,~P_2,~P_2^\ast$ respectively, such that $(\Psi_i,~\Phi_i)_{k_i}=I,~i=1,2,$ ($I$ is the identity matrix), and %the dual bases satisfy
	$$
	\dot{\Phi }_i=\Phi_i B_i ~~~~\textrm{and}~~~~-\dot{\Psi_i }=B_i\Psi_i ,~i=1,2,
	\quad \textrm{with} \quad B_1=0,~B_2=\mathrm{diag}(\mathrm{i}\omega_0,-\mathrm{i}\omega_0)\,.
	$$
	We know from \cite{Hale1977}  that
\begin{equation}\label{eq5+6}
\begin{split}
	\phi_1(\theta)&\equiv\phi_1(0), \;\; \phi_2(\theta)=\phi_2(0)e^{\mathrm{i}\omega_0\theta}, \;\; \theta \in[-r,0],\\
	\psi_1(\theta)&\equiv\psi_1(0), \;\; \psi_2(s)=\psi_2(0)e^{-\mathrm{i}\omega_0s}, \;\; s \in[0,r].
    \end{split}
	\end{equation}
Now we use the definitions above to decompose $\mathcal{C}$ by $\Lambda$:
	 $$\mathcal{C}=\mathcal{P}\oplus\mathcal{Q},~~\mathcal{P}=\textrm{Im}\pi,~~\mathcal{Q}=\textrm{Ker}\pi,$$
	where $\mathrm{dim}\mathcal{P}=3$ and $\pi:\mathcal{C}\rightarrow \mathcal{P}$ is the projection defined by
	\begin{equation}\label{eq5+}
	\pi\phi=\sum_{i=1,2}\Phi_i(\Psi_i,\langle \phi(\cdot),\beta_{k_i}\rangle)_{k_i}\beta_{k_i}\,.
	\end{equation}
	
We project the infinite-dimensional flow on $\mathcal{C}$ to the one on a finite-dimensional manifold $\mathcal{P}$. Following the ideas in
\cite{Faria2000}, we consider the enlarged phase space $B\mathcal{C}$ introduced in ${\bf (H2)}$. This space can be identified as $\mathcal{C}\times X^m,$ with elements  in the form $\phi =\varphi +X_0c$, where $\varphi \in \mathcal{C},~c\in \mathbb{R}^n$,
and $X_0$ is the $m\times m$ matrix-valued function defined by $X_0(\theta)=0$ for $\theta\in [-r,~0)$ and $X_0(0)=I.$
	
	In $B\mathcal{C}$, we consider an extension of the infinitesimal generator, still denoted by $A$,
\begin{equation}\label{eq5+1}
A:\mathcal{C}_0^1\subset B\mathcal{C}\rightarrow B\mathcal{C},~~A\phi=\dot{\phi}+X_0[L_0\phi +D_0\Delta\phi(0)-\dot{\phi}(0)],\end{equation}
defined on $\mathcal{C}_0^1\triangleq \{\phi\in\mathcal{C}:~\dot{\phi}\in\mathcal{C},~\phi(0)\in \text{dom}(\Delta)\}$, and $(\cdot, \cdot)_{k_i}$ can be continuously defined by the same expression \eqref{eq5}, $i=1,2$, on $C^\ast \times BC$, where
$$BC = \left\{\psi : [-r, 0]\rightarrow \mathbb{C}^m |\psi \; \text{ is continuous on }\; [-r, 0), ~ \ds\lim_{\theta \rightarrow 0-} \psi(\theta) \; \text{exists}  \right\}.$$
Thus it is easy to see that $\pi$, as defined in \eqref{eq5+}, is extended to a continuous
projection (which we still denote by $\pi$) $\pi : B\mathcal{C} \rightarrow \mathcal{P}$. In particular, for $c \in X^m$ we have
\begin{equation}\label{eq5++}
\pi(X_0c)=\sum_{i=1,2}\Phi_i\Psi_i(0)\langle c,\beta_{k_i}\rangle\beta_{k_i}\,.
\end{equation}
The projection $\pi$ leads to the topological decomposition
\begin{equation}\label{eq5+2}
B\mathcal{C}=\mathcal{P}\oplus\text{Ker}\pi,
\end{equation}
with the property $\mathcal{Q}\subsetneqq \text{Ker}\pi$.
	
In the space $B\mathcal{C},$ \eqref{eq301} becomes an abstract ODE
\begin{equation}\label{eq6}
\frac{\mathrm{d}v}{\mathrm{d}t} = Av + X_0F(v,\alpha),
\end{equation}
where
\begin{equation}\label{F}
F(v, \alpha )= (L_{\alpha}- L_0)v+(D(\alpha)- D(0))\Delta v(0)  + G(v, \alpha ),
\end{equation}
for $v\in \mathcal{C}$, $\alpha\in V$.
	%$A$ is defined by
	%$$
	%A:~C^1\rightarrow
	%BC,~A\varphi=\dot{\varphi}+X_0[L_0\varphi-\dot{\varphi}(0)].
	%$$
	%The definition of the continuous projection
	%$$\pi:BC\rightarrow P,\ \pi(\varphi+X_0\alpha)=\Phi[(\Psi,\varphi)+\Psi(0)\alpha]$$
	%allows us to decompose the enlarged phase space by $\Lambda$ as
	%$BC=P\oplus \mathrm{Ker}\pi.$
We decompose $v\in\mathcal{C}_0^1$ according to \eqref{eq5+2} by
\begin{equation*}
  v(t)=\phi_1z_1(t)\beta_{k_1}+(\phi_2z_2(t)+\bar{\phi}_2\bar{z}_2(t))\beta_{k_2}+y(t),
\end{equation*}
where $z_i(t)=(\psi_i,\langle v(t)(\cdot),\beta_{k_i}\rangle)_{k_i}$ ($i=1,2$), and $y(t)\in\mathcal{C}_0^1\cap\text{Ker}\pi=\mathcal{C}_0^1\cap\mathcal{Q}\triangleq\mathcal{Q}^1$.
	
Since $\pi$ commutes with $A$ in $\mathcal{C}_0^1$, we see that in $B\mathcal{C}$, the abstract ODE \eqref{eq6} is equivalent to a system of ODEs:
	\begin{equation} \label{eq7}
\begin{split}
	\dot{z}_1  &=  \psi_1 (0)\langle F(\phi_1z_1\beta_{k_1}+(\phi_2z_2+\bar{\phi}_2\bar{z}_2)\beta_{k_2}+y, \alpha ),\beta_{k_1}\rangle , \\
	\dot{z}_2  &=  i\omega_0z_2+\psi_2 (0)\langle F(\phi_1z_1\beta_{k_1}+(\phi_2z_2+\bar{\phi}_2\bar{z}_2)\beta_{k_2}+y, \alpha ),\beta_{k_2}\rangle,\\
	\dot{\bar{z}}_2  &=  -i\omega_0\bar{z}_2+\bar{\psi}_2 (0)\langle F(\phi_1z_1\beta_{k_1}+(\phi_2z_2+\bar{\phi}_2\bar{z}_2)\beta_{k_2}+y, \alpha ),\beta_{k_2}\rangle, \\
	\dfrac{\mathrm{d}}{\mathrm{d}t}y  &= A_{1}y+ (I-\pi )X_0F(\phi_1z_1\beta_{k_1}+(\phi_2z_2+\bar{\phi}_2\bar{z}_2)\beta_{k_2}+y, \alpha),
\end{split}
\end{equation}
for $z=(z_1,z_2,\bar{z}_2)\in P\subset\mathbb{C}^3$,  $y \in \mathcal{Q}^1 \subset \mathrm{Ker}\pi$, where $A_{1}$ is the restriction of $A$ as an operator from $\mathcal{Q}^1$ to the Banach space $\mathrm{Ker}\pi$: $A_1:\mathcal{Q}^1\subset \text{Ker}\pi\rightarrow \text{Ker}\pi,~A_1\phi=A\phi$ for $\phi\in \mathcal{Q}^1$.
	
	 %&&&&&&&&&&&&&&&&&&&&&&&&&&&&&&&&&&&&&&&&&&&&&&&&&&&&&&&&&&&&&&&&&&&&&&&&&&&&&&&&&&&
	%                  New section 3                                                      &
	 %&&&&&&&&&&&&&&&&&&&&&&&&&&&&&&&&&&&&&&&&&&&&&&&&&&&&&&&&&&&&&&&&&&&&&&&&&&&&&&&&&&&
\section{The formulas of second  and third terms in the normal forms}
		
In this section, we present the formulas of the second order and third order terms of the normal form of \eqref{eq301}, and the proofs will be postponed to  Section \ref{proof}.

 First we neglect the dependence on the higher order ($\geq2$) terms of small parameters
$\alpha_1, \alpha_2$ in the third order terms of  the normal forms of \eqref{eq301}. By doing the Taylor expansion formally for the operator $L(\alpha)$ and diagonal matrix $D(\alpha)$ at $\alpha=0$, we have
\begin{equation}\label{eq6+-}
\begin{split}
\ds L(\alpha)v=&L(0)v+\frac{1}{2}L_1(\alpha)v+\cdots,\;\;\; \text{for}~v\in\mathcal{C},\\
\ds D(\alpha)=&D(0)+\frac{1}{2}D_1(\alpha)+\cdots,
\end{split}
\end{equation}
where $L_1:~V\rightarrow L(\mathcal{C},\mathbb{R}^m)$, and $D_1:V\rightarrow \mathbb{R}^m\times \mathbb{R}^m$ are linear.
As in \cite{Hassard1981}, we write $G$  in \eqref{F} in the form of
\begin{equation}\label{eq6+}
G(v,0)=\frac{1}{2!} Q(v,v)+\frac{1}{3!} C(v,v,v)+O(|v|^4),~v\in\mathcal{C},
\end{equation}
where $Q,C$ are symmetric multilinear forms. For simplicity, we  also write $Q(X,Y)$ as $Q_{XY}$, $Q_XY$ or $Q_YX$, and $C(X,Y,Z)$ as $C_{XYZ}$.

%&&&&&&&&&&&&&&&&&&&&&&&&
%&&&&&&&&&&&&&&&&&&&&&&&&
	The formulas of the third order normal form of \eqref{eq301} are given in the following theorem, and the detailed proof of the theorem is provided in Section 6.	
\begin{theorem}\label{g3}
Assume that {\bf (H1)}-{\bf (H4)} are satisfied. Ignoring the effect of the perturbation parameters in high-order items $(\geq 3)$,  then the normal forms of \eqref{eq301} restricted  on the center manifold up to the third order  is
\begin{equation}\label{eq441}
\dot{z}=Bz+\frac{1}{2}g_2^1(z,0,\alpha)+\frac{1}{3!}g_3^1(z,0,\alpha)+h.o.t.,
\end{equation}
%	 with
%	\ba \label{eq44+1}	
%	\dfrac{1}{3!}g^1_3(z,0,0)=g_{31}(z)+g_{32}(z)+g_{33}(z)+g_{34}(z).
%	\ea
%	and $g_{31}(z),g_{32}(z),g_{33}(z),g_{34}(z)$ are given by \eqref{g31}, \eqref{g32}, \eqref{g33}, \eqref{g34}.
%	Assume that $(\mathrm{A}1),~(\mathrm{A}2),~(\mathrm{H}1)$ and $(\mathrm{H}2)$ are satisfied. $(\text{Im}(M_3^1))^c$ spanned by \eqref{eq6+2-1}.equivalent to
or equivalently
\begin{equation}\label{third}
\begin{split}
\dot{z}_1=&a_{1}(\alpha)z_1+a_{11}z_1^2+a_{23}z_2\bar{z}_2++a_{111}z_1^3+ a_{123}z_1z_2\bar{z}_2+h.o.t.,\\
\dot{z}_2=&\mathrm{i}\omega_0z_2+b_{2}(\alpha)z_2+b_{12}z_1z_2+b_{112}z_1^2z_2+b_{223}z_2^2\bar{z}_2+ h.o.t.,\\ \dot{\bar{z}}_2=&-\mathrm{i}\omega_0\bar{z}_2+\overline{b_{2}(\alpha)}\bar{z}_2+\overline{b_{12}}z_1\bar{z}_2
+\overline{b_{112}}z_1^2\bar{z}_2+\overline{b_{223}}z_2\bar{z}_2^2+h.o.t.,
\end{split}
\end{equation}
Here $a_{ij}$, $b_{ij}$, $a_{ijk}$, $b_{ijk}$ are given by
\begin{equation}\label{eq03-5+++}
\begin{split}
	a_{11}=&\frac{1}{2}\psi_1 (0)Q_{\phi_1\phi_1}\langle\beta_{k_1}^2,\beta_{k_1}\rangle,\;
	a_{23}=\psi_1 (0)Q_{\phi_2\bar{\phi}_2}\langle\beta_{k_2}^2,\beta_{k_1}\rangle,\;
	b_{12}=\psi_2 (0)Q_{\phi_1\phi_2}\langle\beta_{k_1}\beta_{k_2},\beta_{k_2}\rangle,
\end{split}
	\end{equation}
\begin{equation}\label{eq03-5++++}
\begin{split}
	a_{111}=&\ds\frac{1}{6}\psi_1 (0)C_{\phi_1\phi_1\phi_1}\langle\beta_{k_1}^3,\beta_{k_1}\rangle+\psi_1 (0)\langle Q_{\phi_1 h_{200}}\beta_{k_1},\beta_{k_1}\rangle\\
&+\frac{1}{2i\omega_0}\psi_1 (0)[-Q_{\phi_1\phi_2}\psi_2 (0)+Q_{\phi_1\bar{\phi}_2}\bar{\psi}_2 (0)]Q_{\phi_1\phi_1}\langle\beta_{k_1}\beta_{k_2},\beta_{k_1}\rangle
	\langle\beta_{k_1}^2,\beta_{k_2}\rangle,\\
	a_{123}=&\ds\psi_1 (0)C_{\phi_1\phi_2\bar{\phi}_2}\langle\beta_{k_1}\beta_{k_2}^2,\beta_{k_1}\rangle+
	\frac{1}{i\omega_0}\psi_1 (0)\{[-Q_{\phi_1\phi_2}\psi_2 (0)+Q_{\phi_1\bar{\phi}_2}\bar{\psi}_2 (0)]Q_{\phi_2\bar{\phi}_2}\langle\beta_{k_1}\beta_{k_2},\beta_{k_1}\rangle\\
	&\ds\langle\beta_{k_2}^2,\beta_{k_2}\rangle
	+\frac{1}{2}[-Q_{\phi_2\phi_2}\psi_2 (0)Q_{\phi_1\bar{\phi}_2}+Q_{\bar{\phi}_2\bar{\phi}_2}\bar{\psi}_2 (0)Q_{\phi_1\phi_2}]\langle\beta_{k_1}\beta_{k_2},\beta_{k_2}\rangle\langle\beta_{k_2}^2,\beta_{k_1}\rangle]\}\\
	&+\psi_1 (0)(\langle Q_{\phi_1 h_{011}}\beta_{k_1},\beta_{k_1}\rangle +
	\langle Q_{\phi_2 h_{101}}\beta_{k_2},\beta_{k_1}\rangle+\langle Q_{\bar{\phi}_2 h_{110}}\beta_{k_2},\beta_{k_1}\rangle),
\end{split}
	\end{equation}
\begin{equation}\label{eq03-5+++++}
\begin{split}
	b_{112}=&\ds\frac{1}{2}\psi_2 (0)C_{\phi_1\phi_1\phi_2}\langle\beta_{k_1}^2\beta_{k_2},\beta_{k_2}\rangle+ \frac{1}{2i\omega_0}\psi_2(0)\{[2Q_{\phi_1\phi_1}\psi_1(0)\langle\beta_{k_1}\beta_{k_2},\beta_{k_1}\rangle\langle\beta_{k_1}^2,\beta_{k_2}\rangle\\
	&+Q_{\phi_1\bar{\phi}_2}\bar{\psi}_2 (0)\langle\beta_{k_1}\beta_{k_2},\beta_{k_2}\rangle^2]Q_{\phi_1\phi_2} + [-Q_{\phi_2\phi_2}\psi_2(0)+Q_{\phi_2\bar{\phi}_2}\bar{\psi}_2(0)]Q_{\phi_1\phi_1}\langle\beta_{k_1}^2,\beta_{k_2}\rangle\langle\beta_{k_2}^2,\beta_{k_2}\rangle \}\\
&+\psi_2 (0)(\langle Q_{\phi_1 h_{110}}\beta_{k_1},\beta_{k_2}\rangle+
	\langle Q_{\phi_2 h_{200}}\beta_{k_2},\beta_{k_2}\rangle),\\
b_{223}=&\frac{1}{2}\psi_2 (0)C_{\phi_2\phi_2\bar{\phi}_2}\langle\beta_{k_2}^3 ,\beta_{k_2}\rangle+
	\frac{1}{4i\omega_0}\psi_2 (0)\{Q_{\phi_1\bar{\phi}_2}\psi_1 (0)Q_{\phi_2\phi_2}\langle\beta_{k_1}\beta_{k_2},\beta_{k_2}\rangle\langle\beta_{k_2}^2,\beta_{k_1}\rangle +\\&
	\frac{2}{3}Q_{\bar{\phi}_2\bar{\phi}_2}\bar{\psi}_2 (0)Q_{\phi_2\phi_2}\langle\beta_{k_2}^2,\beta_{k_2}\rangle^2
	+[-2Q_{\phi_2\phi_2}\psi_2 (0)+4Q_{\phi_2\bar{\phi}_2}\bar{\psi}_2 (0)]Q_{\phi_2\bar{\phi}_2}\langle\beta_{k_2}^2,\beta_{k_2}\rangle^2\}\\
	&+\psi_2 (0) (\langle Q_{\phi_2 h_{011}}\beta_{k_2},\beta_{k_2}\rangle+\langle Q_{\bar{\phi}_2 h_{020}}\beta_{k_2},\beta_{k_2}\rangle).
	\end{split}
	\end{equation}
respectively,	and $h_{ijk}~(i+j+k=2,~i,j,k\in \mathbb{N}_0)$ are determined by \eqref{eq16+2} and \eqref{eq16+3-}.
\end{theorem}	

\begin{remark}
For the normal form up to the third order, we only need to calculate the eigenvectors which are given by \eqref{eq5+3}, the linear parts $L_1(\alpha)$ and $D_1(\alpha)$ in \eqref{eq6+-}, and the multilinear forms $Q$ and $C$ which are given in \eqref{eq6+}.
\end{remark}

For the reduced system \eqref{third}, the bifurcation structure can be distinguished into the two main types: Hopf-transcritical type and Hopf-pitchfork type, which we will discuss separately below.
    %&&&&&&&&&&&&&&&&&&&&&&&&&&&&&&&&&&&&&&&&&&&&&&&&&&&&&&&&&&&&&&&&&&&&&&&&&&&&&&&&&&&
	%                  New section 3-1                                                     &
	 %&&&&&&&&&&&&&&&&&&&&&&&&&&&&&&&&&&&&&&&&&&&&&&&&&&&&&&&&&&&&&&&&&&&&&&&&&&&&&&&&&&&
\subsection{Hopf-transcritical type}
Following \cite{JiangW2010}, we have
\begin{definition}\label{def3.3}
Assume that {\bf (H1)}-{\bf (H4)}
are satisfied, $a_{11}\ne 0$, $a_{23}\ne 0$, ${\rm Re}(b_{12})\ne 0$, and $a_{11}-{\rm Re}(b_{12})\ne 0$.
%where
  %\begin{equation}\label{eq1+}
	%\begin{array}{rlc}
	%a_1=&\text{Re}(\psi_2 (0)Q_{\phi_1\phi_2}\langle\beta_{k_1}\beta_{k_2},\beta_{k_2}\rangle),\\
	%b_1=&\psi_1 (0)Q_{\phi_2\bar{\phi}_2}\langle\beta_{k_2}^2,\beta_{k_1}\rangle,\\
	%b_2=&\frac{1}{2}\psi_1 (0)Q_{\phi_1\phi_1}\langle\beta_{k_1}^2,\beta_{k_1}\rangle. \,
	%\end{array}
	%\end{equation}
Then we say that a  Hopf-steady state bifurcation with Hopf-transcritical type  occurs for \eqref{eq301} (or referred as a Hopf--transcritical bifurcation) at the trivial equilibrium when $\alpha=0$.
\end{definition}
Moreover we adopt the same coordinate transformation (23) and (25) in
\cite{JiangW2010}, then \eqref{third} can be reduced to the planar system (see \cite{JiangW2010})
\begin{equation}\label{eq473}
\begin{split}
\dot{r}=& r(\varepsilon_1(\alpha)+az+cr^2+dz^2),\\
\dot{z}=& \varepsilon_2(\alpha)z+br^2-z^2+er^2z+fz^3,
\end{split}
\end{equation}
where
\begin{equation*}
\begin{split}
&\varepsilon_1(\alpha)=\mathrm{Re}(b_2(\alpha)),~\varepsilon_2(\alpha)=a_1(\alpha),~
a=-\frac{\mathrm{Re}(b_{12})}{a_{11}}, ~b=-\mathrm{sign}(a_{11}a_{23}),\\
&c=\frac{\mathrm{Re}(b_{223})}{|a_{11}a_{23}|},~d=
\frac{\mathrm{Re}(b_{112})}{{a_{11}}^2},~e=\frac{a_{123}}{|a_{11}a_{23}|},~f
=\frac{a_{111}}{{a_{11}}^2}.
\end{split}
\end{equation*}
Now \eqref{eq473}  has the same form as (36) in \cite{JiangW2010}. By \cite{Guck1983} and \cite{JiangW2010},  there are four different topological structure for \eqref{eq473} with the Hopf-transcritical bifurcation depending on the signs of $a$ and $b$:
	$$
	\begin{array}{llll}
	\mathrm{Case ~I}: &b=1, a>0; &   ~~\mathrm{ Case ~II}: &b=1, a<0;\\
	\mathrm{Case~III}: &b=-1, a>0; & ~~\mathrm{Case ~IV}: &b=-1, a<0.\\
	\end{array}
	$$
The results  in \cite{JiangW2010} can be directly applied to analyze the
equation \eqref{eq473} and  the dynamical properties for original system \eqref{eq301} can be revealed with the help of the analysis in \cite [Section 4] {AJ}.
%&&&&&&&&&&&&&&&&&&&&&&&&&&&&&&&&&&&&&&&&&&&&&&&&&&&&&&&&&&&&&&&&&&&&&&&&&&&&&&&&&&&
	%                  New section 3-2                                                     &
	 %&&&&&&&&&&&&&&&&&&&&&&&&&&&&&&&&&&&&&&&&&&&&&&&&&&&&&&&&&&&&&&&&&&&&&&&&&&&&&&&&&&&
\subsection{Hopf-pitchfork type}
Following \cite{Guck1983}, we have
\begin{definition}
Assume that {\bf (H1)}-{\bf (H4)}
are satisfied, $a_{11}=a_{23}=b_{12}=0$,  $a_{111}\ne 0$, $a_{123}\ne 0$, $\mathrm{Re}(b_{112})\ne 0$, $\mathrm{Re}(b_{223})\neq 0$, and $a_{111}\mathrm{Re}(b_{223})-a_{123}\mathrm{Re}(b_{112})\neq 0$. Then  we say that a Hopf-steady state bifurcation with Hopf-pitchfork type  occurs for
\eqref{eq301} at the trivial equilibrium when $\alpha=0$.
\end{definition}
By using the same coordinate transformation in
(23) of \cite{JiangW2010}  and the coordinate transformation
\begin{equation}\label{tran-2}
\sqrt{|\mathrm{Re}b_{223}|}~r\rightarrow r, ~\sqrt{|a_{111}|}~z \rightarrow z,~\mathrm{sign}(\mathrm{Re}(b_{223})) t\rightarrow t,\end{equation}
we obtain a planar system (see \cite{Guck1983} and \cite{WangJ2010}):
\begin{equation}\label{eq473-2}
\begin{split}
\dot{r}=& r(\varepsilon_1(\alpha)+r^2+b_0z^2),\\
\dot{z}=& z(\varepsilon_2(\alpha)+c_0r^2+d_0z^2),
\end{split}
\end{equation}
where
\begin{equation*}
  \begin{split}
&\varepsilon_1(\alpha)=\mathrm{Re}(b_2(\alpha))\mathrm{sign}(\mathrm{Re}(b_{223})),
~\varepsilon_2(\alpha)=a_1(\alpha)\mathrm{sign}(\mathrm{Re}(b_{223})),\\
&b_0=\ds\frac{\mathrm{Re}(b_{112})}{|a_{111}|}\mathrm{sign}(\mathrm{Re}(b_{223})),
~c_0=\frac{a_{123}}{|\mathrm{Re}(b_{223})|}\mathrm{sign}(\mathrm{Re}(b_{223})),
~d_0=\mathrm{sign}(a_{111}\mathrm{Re}(b_{223})).
\end{split}
\end{equation*}

For system \eqref{eq473-2}, there are possibly four equilibrium points as follows:
\begin{equation}
%\begin{aligned}{ll}
\begin{array}{ll}
E_1 = (0,0),  %\hspace{4cm}
&\mathrm{for~all}\;\; \varepsilon_1, \varepsilon_2,\\
E_2 = (\sqrt{-\varepsilon_1},0),%\hspace{3.25cm}
&\mathrm{for}\; \varepsilon_1<0,\\
E_3^{\pm} = (0,{\pm}\ds\sqrt{-\frac{\varepsilon_2}{d_0}}),%\hspace{3.15cm}
&\mathrm{for}\;\; \ds\frac{\varepsilon_2}{d_0}<0,\\
E_4^{\pm} = (\ds\sqrt{\frac{b_0\varepsilon_2-d_0\varepsilon_1}{d_0-b_0c_0}},\ds{\pm}\sqrt{\frac{c_0\varepsilon_1-\varepsilon_2}{d_0-b_0c_0}}),
%\hspace{0.75cm}
&\mathrm{for}\;\ds\frac{b_0\varepsilon_2-d_0\varepsilon_1}{d_0-b_0c_0},\ds\frac{c_0\varepsilon_1-\varepsilon_2}{d_0-b_0c_0}>0.
\end{array}
%\end{aligned}
\end{equation}
Based on \cite [\S7.5]{Guck1983}, by the different signs of
$b_0,c_0,d_0,d_0-b_0c_0$ in Table \ref{tab1},  Eq. \eqref{eq473-2} has twelve distinct types of unfoldings, which are twelve essentially distinct types of phase portraits and bifurcation diagrams. The results  in \cite{Guck1983} can be directly applied to analyze the
equation \eqref{eq473-2} and  the dynamical properties for original system \eqref{eq301} can be revealed with the help of the analysis in \cite [Section 4] {AJ}. For the convenience of application to the example in Section 5, we show the bifurcation digram in parameters $(\varepsilon_1,\varepsilon_2)$  and phase portraits of Case $\mathrm{III}$ of Hopf-pitchfork type (see Table \ref{tab1}) in Figure \ref{fig1}.

\begin{table}
\begin{center}
\begin{tabular}{|c|cccccccccccc|}
\hline
~~~Case~~~ &$\mathrm{I}a$ &$ \mathrm{I}b$ & $\mathrm{II} $
&  $ \mathrm{III}$ &$ \mathrm{IV}a$&$ \mathrm{IV}b$& $ \mathrm{V}$
&$ \mathrm{VI}a$ &  $ \mathrm{VI}b$ & $ \mathrm{VII}a$ & $ \mathrm{VII}b$ &$ \mathrm{VIII}$\\
\hline $d_0$ & $+1$ & $+1$ & $+1$ & $+1$ & $+1$ & $+1$ & $-1$& $-1$& $-1$& $-1$& $-1$& $-1$\\
 $ b_0$ & $+$ & $+$& $+$& $-$& $-$& $-$& $+$& $+$& $+$& $-$ & $-$& $-$\\
$c_0$ & $+$& $+$& $-$& $+$& $-$& $-$& $+$& $-$& $-$& $+$& $+$& $-$\\
$ d_0-b_0c_0$ & $+$& $-$& $+$& $+$& $+$& $-$& $-$& $+$& $-$& $+$& $-$& $-$\\
%\hline $ \mathrm{Hopf}$ & &&&&&&\mathrm{occurs}&\mathrm{occurs}&&&\\
\hline
\end{tabular}
\end{center}
\caption{The twelve unfoldings of \eqref{eq473-2} \cite {Guck1983}}\label{tab1}
\end{table}

\begin{figure}[htbp]
	\centering
\subfigure[]{
		\begin{minipage}{0.4\linewidth}
			\centering
			\includegraphics[scale=0.22]{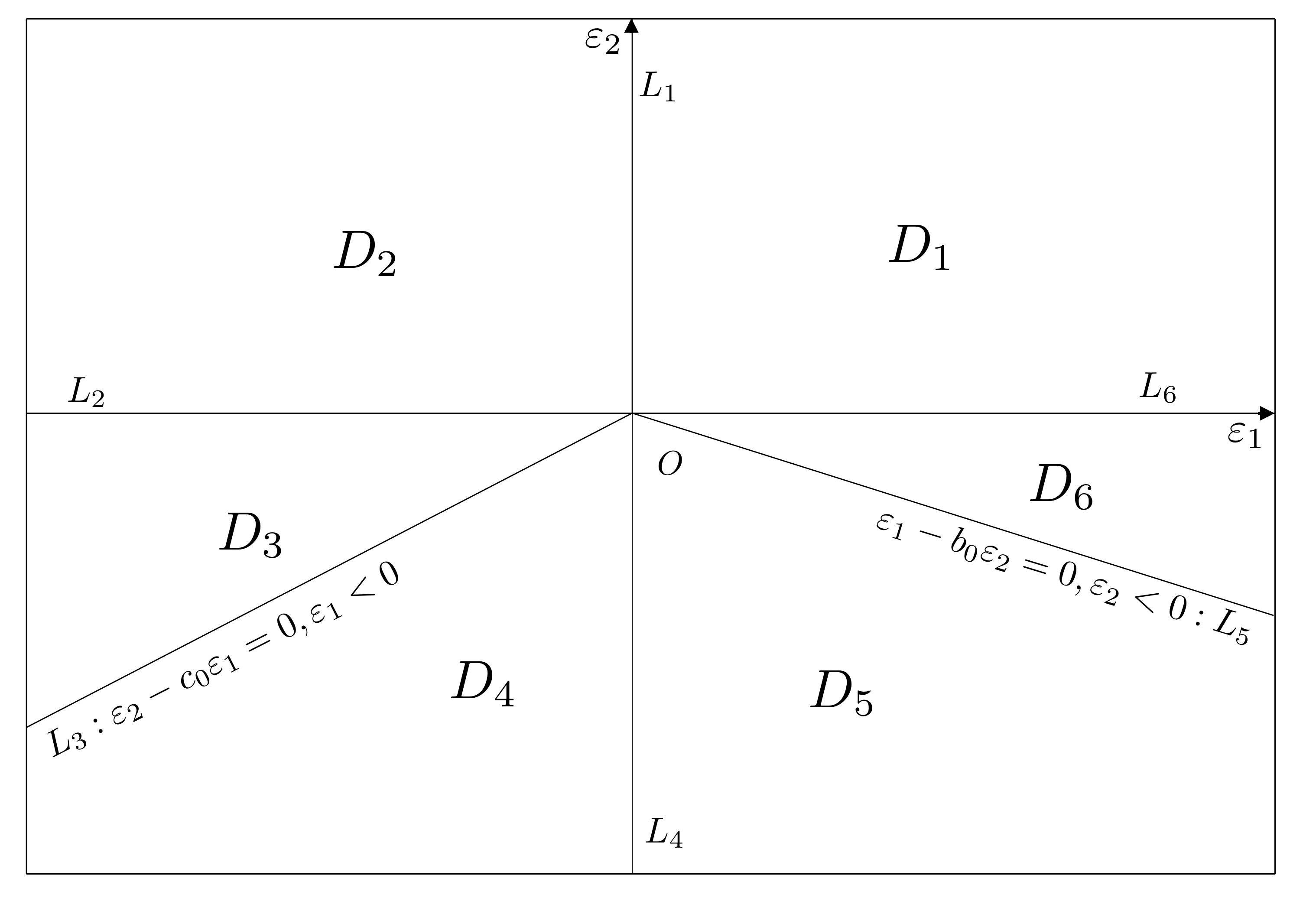}
	\end{minipage}}
\subfigure[]{
		\begin{minipage}{0.56\linewidth}
			\centering
			\includegraphics[scale=0.34]{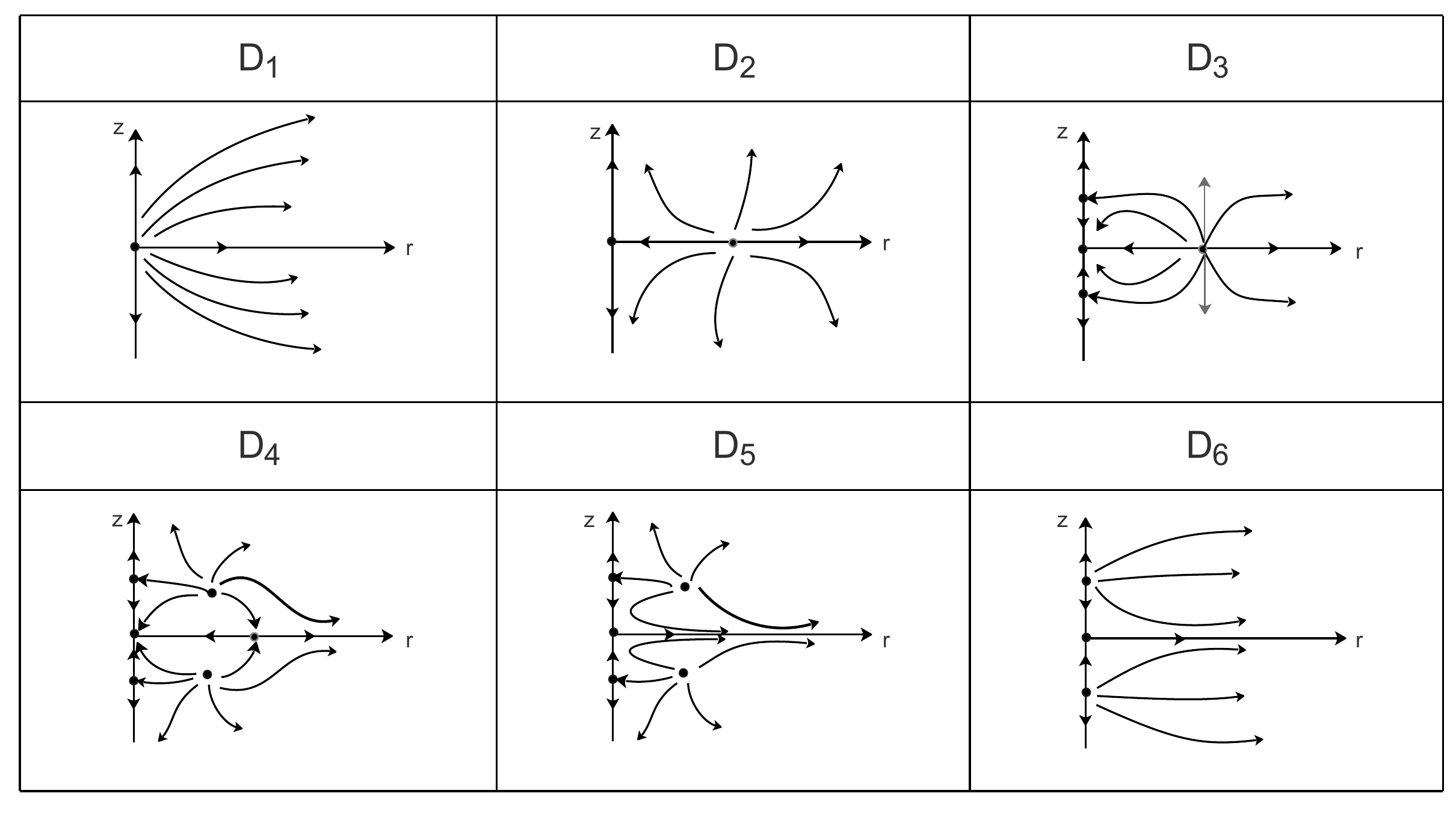}
		\end{minipage}}
\caption{The bifurcation set (a) and the phase portraits (b) for Case $ \mathrm{III}$ of Hopf-pitchfork type. \label{fig1} }
\end{figure}

In the next section, we will give the exact expressions of  $h_{ijk}$ in \eqref{eq03-5+++}, \eqref{eq03-5++++} and \eqref{eq03-5+++++} under the Neumann boundary condition and $\Omega =(0,l\pi)$ for $l>0$.

\section{Explicit formulas for Neumann boundary conditions}
	
In this section, we give more explicit formulas of coefficients $h_{ijk}$ in \eqref{eq03-5+++}, \eqref{eq03-5++++} and \eqref{eq03-5+++++} of the normal form \eqref{third} restricted on the centre manifold in the case of spatial dimension $n=1$ and $\Omega = (0,l\pi)$ for some $l>0$, and we consider \eqref{eq301} with Neumann boundary condition.
	
It is well known that the eigenvalue problem
$$
\beta''+\mu\beta=0,~~~x\in(0,l\pi),~~~\beta'(0)=\beta'(l\pi)=0,$$
has eigenvalues $\mu_n=n^2/l^2$ for $n\in \N_0$ with corresponding normalized eigenfuctions
$$\beta_0(x)=1,\;\;\beta_n(x)=\sqrt{2}\cos\frac{n}{l}x,~~~n\in \N,$$
	where $\langle\beta_m(x),\beta_n(x) \rangle =\frac{1}{l\pi}\int_0^{l\pi}\beta_m(x)\beta_n(x)\mathrm{d}x=\delta_{mn}$.  In what follows,  we denote  $$h_q^{k}=\langle h_q,\beta_{k} \rangle,~~q\in \mathbb{N}_0^3,~~|q|=2,~~k\in\mathbb{N}_0.$$
	
According to the different situations of $k_1,~k_2$ in ($\bf{H4}$), for the convenience of applications, we will give the exact formulas of $a_{11},~a_{23},~a_{111},~a_{123},~b_{12},~b_{112},~b_{223}$ in \eqref{third} for the following five cases.
	
\noindent{\bf Case (1) $k_1=k_2=0$.}
	
In this case,
$$\begin{array}{rl}
\langle\beta_{k_1}^2,\beta_{k_1}\rangle=& \langle\beta_{k_2}^2,\beta_{k_1} \rangle=\langle\beta_{k_1}\beta_{k_2},\beta_{k_2} \rangle=\langle\beta_{0},\beta_{0}\rangle=1,\\
\langle\beta_{k_1}^2,\beta_{k_2}\rangle=& \langle\beta_{k_2}^2,\beta_{k_2} \rangle=\langle\beta_{k_1}\beta_{k_2},\beta_{k_1} \rangle=\langle\beta_{0},\beta_{0}\rangle=1,\\
\langle\beta_{k_1}^3,\beta_{k_1}\rangle=& \langle\beta_{k_2}^3,\beta_{k_2} \rangle=\langle\beta_{k_1}\beta_{k_2}^2,\beta_{k_1} \rangle=\langle\beta_{k_1}^2\beta_{k_2},\beta_{k_2} \rangle=\langle\beta_{0},\beta_{0}\rangle=1,\\
\langle Q_{\phi h_q}\beta_{k_i},\beta_{k_j} \rangle=&Q_{\phi} h_q^0,~~ \phi \in \{ \phi_1, \phi_2, \bar{\phi}_2 \},~~i,j=1,2,~~q\in \mathbb{N}_0^3,~~|q|=2.
\end{array}$$
By \eqref{eq16+2} and \eqref{eq16+3-}, $h_q^{k_i}=h_q^0$ for $i=1,2$ and $q\in \mathbb{N}_0^3$ with $|q|=2$, which is needed in \eqref{eq03-5+++}, \eqref{eq03-5++++} and \eqref{eq03-5+++++}, are as follows:
\begin{equation}\label{eq03-1}
\begin{array}{rlc}
	 h_{200}^0(\theta)=&\frac{1}{2}[\theta\phi_1(0)\psi_1(0)+\frac{1}{i\omega_0}(\phi_2(\theta)\psi_2(0)-\bar{\phi}_2(\theta)\bar{\psi}_2(0))]Q_{\phi_1\phi_1}+E_{200},\\
	 h_{011}^0(\theta)=&[\theta\phi_1(0)\psi_1(0)+\frac{1}{i\omega_0}(\phi_2(\theta)\psi_2(0)-\bar{\phi}_2(\theta)\bar{\psi}_2(0))]Q_{\phi_2\bar{\phi}_2}+E_{011},\\
	 h_{020}^0(\theta)=&-\frac{1}{2i\omega_0}[\frac{1}{2}\phi_1(0)\psi_1(0)+\phi_2(\theta)\psi_2(0)+\frac{1}{3}\bar{\phi}_2(\theta)\bar{\psi}_2(0))]Q_{\phi_2\phi_2}+E_{020}e^{2i\omega_0\theta},\\
	 h_{110}^0(\theta)=&\frac{1}{i\omega_0}[-\phi_1(0)\psi_1(0)+i\omega_0\theta\phi_2(\theta)\psi_2(0)-\frac{1}{2}\bar{\phi}_2(\theta)\bar{\psi}_2(0))]Q_{\phi_1\phi_2}+E_{110}e^{i\omega_0\theta},\\
	 h_{002}^0(\theta)=&\overline{h_{020}^0(\theta)},~~~h_{101}^0(\theta)=\overline{h_{110}^0(\theta)},
\end{array}
\end{equation}
where $\theta\in[-r,0]$ and the constant vectors $E_q$ for $q\in \mathbb{N}_0^3$ with $|q|=2$, satisfy the following conditions
\begin{equation}\label{eq03-1+}
\begin{array}{rlc}
[\int_{-r}^0\mathrm{d}\eta_0(\theta)]E_{200}=&\frac{1}{2}[-I+(I-\int_{-r}^0\theta \mathrm{d}\eta_0(\theta))\phi_1(0)\psi_1(0)]Q_{\phi_1\phi_1},\\

[\int_{-r}^0 \mathrm{d} \eta_0(\theta)]E_{011}=&[-I+(I-\int_{-r}^0\theta \mathrm{d} \eta_0(\theta))\phi_1(0)\psi_1(0)]Q_{\phi_2\bar{\phi}_2},\\

E_{020}=&\frac{1}{2}[2i\omega_0I-\int_{-r}^0e^{2i\omega_0\theta} \mathrm{d} \eta_0(\theta)]^{-1}Q_{\phi_2\phi_2},\\
%E_{002}=&\frac{1}{2}[-2i\omega_0I-\int_{-r}^0e^{-2i\omega_0\theta} \mathrm{d} \eta_0(\theta)]^{-1}Q_{\bar{\phi}_2\bar{\phi}_2},\\

[i\omega_0I-\int_{-r}^0e^{i\omega_0\theta}\mathrm{d}\eta_0(\theta)]E_{110}=&[I-\phi_2(0)\psi_2(0)+\int_{-r}^0\theta \mathrm{d}\eta_0(\theta)\phi_2(\theta)\psi_2(0)]Q_{\phi_1\phi_2}.
%[-i\omega_0I-\int_{-r}^0e^{-i\omega_0\theta}\mathrm{d}\eta_0(\theta)]E_{101}=&[I-\bar{\phi}_2(0)\bar{\psi}_2(0)+\int_{-r}^0\theta %\mathrm{d}\eta_0(\theta)\bar{\phi}_2(\theta)\bar{\psi}_2(0)]Q_{\phi_1\bar{\phi}_2}.
\end{array}
\end{equation}
Thus we have the following result.

\begin{pro}\label{pro:4.1} For $k_1=k_2=0$ and Neumann boundary condition on spatial domain $\Omega =(0,l\pi),~l>0$, the parameters $a_{11},~a_{23},~a_{111},~a_{123},~b_{12},~b_{112},~b_{223}$ in \eqref{third} are given by
	\begin{equation}\label{eq03-1++}
	\begin{array}{rlc}
	a_{11}=&\frac{1}{2}\psi_1 (0)Q_{\phi_1\phi_1},\;\;
	a_{23}=\psi_1 (0)Q_{\bar{\phi}_2\phi_2},\;\;
	b_{12}=\psi_2 (0)Q_{\phi_1\phi_2},\\
	
	a_{111}=&\frac{1}{6}\psi_1 (0)C_{\phi_1\phi_1\phi_1}+\frac{1}{\omega_0}\psi_1 (0)\mathrm{Re}(\mathrm{i}Q_{\phi_1\phi_2}\psi_2 (0))Q_{\phi_1\phi_1}+\psi_1 (0) Q_{\phi_1} h_{200}^0,\\
	a_{123}=&\psi_1 (0)C_{\phi_1\phi_2\bar{\phi}_2}+\frac{2}{\omega_0}\psi_1 (0)\mathrm{Re}(\mathrm{i}Q_{\phi_1\phi_2}\psi_2 (0))Q_{\phi_2\bar{\phi}_2}\\
	
	&+\frac{1}{\omega_0}\psi_1 (0)\mathrm{Re}(\mathrm{i}Q_{\phi_2\phi_2}\psi_2 (0)Q_{\phi_1\bar{\phi}_2})+\psi_1 (0) (Q_{\phi_1} h_{011}^0+Q_{\phi_2} h_{101}^0+Q_{\bar{\phi}_2} h_{110}^0),\\
	
	b_{112}=&\frac{1}{2}\psi_2 (0)C_{\phi_1\phi_1\phi_2}+\frac{1}{2i\omega_0}\psi_2 (0)\{[2Q_{\phi_1\phi_1}\psi_1 (0) +Q_{\phi_1\bar{\phi}_2}\bar{\psi}_2 (0)]Q_{\phi_1\phi_2}\\
	
	&+[-Q_{\phi_2\phi_2}\psi_2 (0)+Q_{\phi_2\bar{\phi}_2}\bar{\psi}_2 (0)]Q_{\phi_1\phi_1} \}+\psi_2 (0) (Q_{\phi_1} h_{110}^0+Q_{\phi_2} h_{200}^0),\\
	
	b_{223}=&\frac{1}{2}\psi_2 (0)C_{\phi_2\phi_2\bar{\phi}_2}+\frac{1}{4i\omega_0}\psi_2 (0)\{Q_{\phi_1\bar{\phi}_2}\psi_1 (0)Q_{\phi_2\phi_2} +\frac{2}{3}Q_{\bar{\phi}_2\bar{\phi}_2}\bar{\psi}_2 (0)Q_{\phi_2\phi_2}\\&
	
	+[-2Q_{\phi_2\phi_2}\psi_2 (0)+4Q_{\phi_2\bar{\phi}_2}\bar{\psi}_2 (0)]Q_{\phi_2\bar{\phi}_2}\}+\psi_2 (0) (Q_{\phi_2} h_{011}^0+Q_{\bar{\phi}_2} h_{020}^0).
\end{array}
\end{equation}
\end{pro}
	
	%%%
	%%%
	\noindent{\bf Case (2) $k_1=k_2\neq0$.}
	
Here we have
	$$
	\begin{array}{rl}
	\langle\beta_{k_1}^2,\beta_{k_1}\rangle=& \langle\beta_{k_2}^2,\beta_{k_1} \rangle=\langle\beta_{k_1}\beta_{k_2},\beta_{k_2} \rangle=0,\\
	\langle\beta_{k_1}^2,\beta_{k_2}\rangle=& \langle\beta_{k_2}^2,\beta_{k_2} \rangle=\langle\beta_{k_1}\beta_{k_2},\beta_{k_1} \rangle=0,\\
	\langle\beta_{k_1}^3,\beta_{k_1}\rangle=& \langle\beta_{k_2}^3,\beta_{k_2} \rangle=\langle\beta_{k_1}\beta_{k_2}^2,\beta_{k_1} \rangle=\langle\beta_{k_1}^2\beta_{k_2},\beta_{k_2} \rangle=\frac{3}{2},\\
	\langle Q_{\phi h_q}\beta_{k_i},\beta_{k_j} \rangle=&Q_{\phi}( h_q^0+\frac{1}{\sqrt{2}}h_q^{2k_1}),~~ \phi \in \{ \phi_1, \phi_2, \bar{\phi}_2 \},~~i=1,2,~~q\in \mathbb{N}_0^3,~~|q|=2.
	\end{array}
	$$
	It also follows from \eqref{eq16+2} and \eqref{eq16+3-} that $h_q^{k_i}=h_q^0$ for $i=1,2$, $q\in \mathbb{N}_0^3$ with $|q|=2$, are given by
	\begin{equation}\label{eq03-2}
	\begin{array}{rlc}
	h_{200}^0(\theta)\equiv &-\frac{1}{2}[\int_{-r}^0\mathrm{d}\eta _0(\theta)]^{-1}Q_{\phi_1\phi_1},\;\;
	h_{200}^{2k_1}(\theta)\equiv -\frac{1}{2\sqrt{2}}[\int_{-r}^0\mathrm{d}\eta _{2k_1}(\theta)]^{-1}Q_{\phi_1\phi_1},\\
	h_{011}^0(\theta)\equiv &-[\int_{-r}^0\mathrm{d}\eta _0(\theta)]^{-1}Q_{\phi_2\bar{\phi}_2},\;\;
	h_{011}^{2k_1}(\theta)\equiv -\frac{1}{\sqrt{2}}[\int_{-r}^0\mathrm{d}\eta _{2k_1}(\theta)]^{-1}Q_{\phi_2\bar{\phi}_2},\\
	
	 h_{020}^{0}(\theta)=&\frac{1}{2}[2i\omega_0I-\int_{-r}^0e^{2i\omega_0\theta}\mathrm{d}\eta _{0}(\theta)]^{-1}Q_{\phi_2\phi_2}e^{2i\omega_0\theta},\\
	 h_{020}^{2k_1}(\theta)=&\frac{1}{2\sqrt{2}}[2i\omega_0I-\int_{-r}^0e^{2i\omega_0\theta}\mathrm{d}\eta _{2k_1}(\theta)]^{-1}Q_{\phi_2\phi_2}e^{2i\omega_0\theta},\\
	
	h_{110}^0(\theta)=&[i\omega_0I-\int_{-r}^0e^{i\omega_0\theta}\mathrm{d}\eta _{0}(\theta)]^{-1}Q_{\phi_1\phi_2}e^{i\omega_0\theta},\\
	 h_{110}^{2k_1}(\theta)=&\frac{1}{\sqrt{2}}[i\omega_0I-\int_{-r}^0e^{i\omega_0\theta}\mathrm{d}\eta _{2k_1}(\theta)]^{-1}Q_{\phi_1\phi_2}e^{i\omega_0\theta},\\
	
	h_{002}^0(\theta)=&\overline{h_{020}^0(\theta)},~~~
	 h_{002}^{2k_1}(\theta)=\overline{h_{020}^{2k_1}(\theta)},~~~h_{101}^0(\theta)=\overline{h_{110}^0(\theta)},~~~
	h_{101}^{2k_1}(\theta)=\overline{h_{110}^{2k_1}(\theta)}.
	\end{array}
	\end{equation}
	where $\theta\in[-r,0]$.
	Thus we have the following result.

\begin{pro}\label{pro:4.2} For $k_1=k_2\neq0$ and Neumann boundary condition on spatial domain $\Omega =(0,l\pi),~l>0$, the parameters $a_{11},~a_{23},~a_{111},~a_{123},~b_{12},~b_{112},~b_{223}$ in \eqref{third} are
	\begin{equation}\label{eq03-2++}
	\begin{array}{rlc}
	a_{11}=&a_{23}=b_{12}=0,\\
	
	a_{111}=&\frac{1}{4}\psi_1 (0)C_{\phi_1\phi_1\phi_1}+\psi_1 (0) Q_{\phi_1}( h_{200}^0+\frac{1}{\sqrt{2}}h_{200}^{2k_1}),\\
	
	a_{123}=&\frac{3}{2}\psi_1 (0)C_{\phi_1\phi_2\bar{\phi}_2}+\psi_1 (0) [Q_{\phi_1} (h_{011}^0+\frac{1}{\sqrt{2}}h_{011}^{2k_1})+Q_{\phi_2} (h_{101}^0+\frac{1}{\sqrt{2}}h_{101}^{2k_1})+Q_{\bar{\phi}_2}( h_{110}^0+\frac{1}{\sqrt{2}}h_{110}^{2k_1})],\\
	
	b_{112}=&\frac{3}{4}\psi_2 (0)C_{\phi_1\phi_1\phi_2}+\psi_2 (0) [Q_{\phi_1} (h_{110}^0+\frac{1}{\sqrt{2}}h_{110}^{2k_1})+Q_{\phi_2} (h_{200}^0+\frac{1}{\sqrt{2}}h_{200}^{2k_1})],\\
	
	b_{223}=&\frac{3}{4}\psi_2 (0)C_{\phi_2\phi_2\bar{\phi}_2}+\psi_2 (0) [Q_{\phi_2} (h_{011}^0+\frac{1}{\sqrt{2}}h_{011}^{2k_1})+Q_{\bar{\phi}_2} (h_{020}^0+\frac{1}{\sqrt{2}}h_{020}^{2k_1})].
	\end{array}
	\end{equation}
    \end{pro}
	%%%
	%%%
	
	\noindent{\bf Case (3) $k_2=0$, $k_1\neq 0$.}
	
Here we have
	$$
	\begin{array}{rl}
	\langle\beta_{k_1}^2,\beta_{k_1}\rangle=& \langle\beta_{k_2}^2,\beta_{k_1} \rangle=\langle\beta_{k_1}\beta_{k_2},\beta_{k_2} \rangle=0,\\
	\langle\beta_{k_1}^2,\beta_{k_2}\rangle=& \langle\beta_{k_2}^2,\beta_{k_2} \rangle=\langle\beta_{k_1}\beta_{k_2},\beta_{k_1} \rangle=1,\\
	
	\langle\beta_{k_1}^3,\beta_{k_1}\rangle=&\frac{3}{2},~ \langle\beta_{k_2}^3,\beta_{k_2} \rangle=\langle\beta_{k_1}\beta_{k_2}^2,\beta_{k_1} \rangle=\langle\beta_{k_1}^2\beta_{k_2},\beta_{k_2} \rangle=1,\\
	
	\langle Q_{\phi_1 h_{200}}\beta_{k_1},\beta_{k_1}\rangle=&Q_{\phi_1} ( h_{200}^{0}+\frac{1}{\sqrt{2}}h_{200}^{2k_1}),~~~\langle Q_{\phi_1 h_{011}}\beta_{k_1},\beta_{k_1}\rangle=Q_{\phi_1} ( h_{011}^{0}+\frac{1}{\sqrt{2}}h_{011}^{2k_1}),\\
	
	\langle Q_{\phi_2 h_{101}}\beta_{k_2},\beta_{k_1}\rangle=&Q_{\phi_2}  h_{101}^{k_1},~~~~~\langle Q_{\bar{\phi}_2 h_{110}}\beta_{k_2},\beta_{k_1}\rangle=Q_{\bar{\phi}_2}  h_{110}^{k_1},\\
	
	\langle Q_{\phi_1 h_{101}}\beta_{k_1},\beta_{k_2}\rangle=&Q_{\phi_1}  h_{101}^{k_1},~~~~~\langle Q_{\phi_1 h_{110}}\beta_{k_1},\beta_{k_2}\rangle=Q_{\phi_1}  h_{110}^{k_1},\\
	
	\langle Q_{\phi_2 h_{200}}\beta_{k_2},\beta_{k_2}\rangle=&Q_{\phi_2}  h_{200}^{0},~~~~~\langle Q_{\phi_2 h_{011}}\beta_{k_2},\beta_{k_2}\rangle=Q_{\phi_2}  h_{011}^{0},\\
	
	\langle Q_{\bar{\phi}_2 h_{200}}\beta_{k_2},\beta_{k_2}\rangle=&Q_{\bar{\phi}_2}  h_{200}^{0},~~~~~\langle Q_{\bar{\phi}_2 h_{011}}\beta_{k_2},\beta_{k_2}\rangle=Q_{\bar{\phi}_2}  h_{011}^{0},\\
	
	\langle Q_{\bar{\phi}_2 h_{020}}\beta_{k_2},\beta_{k_2}\rangle=&Q_{\bar{\phi}_2}  h_{020}^{0},~~~~~\langle Q_{\phi_2 h_{002}}\beta_{k_2},\beta_{k_2}\rangle=Q_{\phi_2}  h_{002}^{0},
	\end{array}
	$$
	%By (\ref{eq16+2}) and(\ref{eq16+3-}) the obtained results about $ h_{200}^{0}, h_{200}^{2k_1},h_{011}^{0},h_{011}^{2k_1},h_{101}^{k_1},h_{110}^{k_1},h_{020}^{0},h_{002}^{0}$, which  need to be obtained  in (\ref{eq03-5+++}), are as follows.
and
	\begin{equation}\label{eq03-3}
	\begin{array}{rlc}
	h_{200}^0(\theta)=&-\frac{1}{2}[\int_{-r}^0\mathrm{d}\eta _0(\theta)]^{-1}Q_{\phi_1\phi_1}+\frac{1}{2i\omega_0}(\phi_2(\theta)\psi_2(0)-\bar{\phi}_2(\theta)\bar{\psi}_2(0))]Q_{\phi_1\phi_1},\\
	
	h_{200}^{2k_1}(\theta)\equiv &-\frac{1}{2\sqrt{2}}[\int_{-r}^0\mathrm{d}\eta _{2k_1}(\theta)]^{-1}Q_{\phi_1\phi_1},\\
	
	h_{011}^0(\theta)=&-[\int_{-r}^0\mathrm{d}\eta _0(\theta)]^{-1}Q_{\phi_2\bar{\phi}_2}+\frac{1}{i\omega_0}(\phi_2(\theta)\psi_2(0)-\bar{\phi}_2(\theta)\bar{\psi}_2(0))]Q_{\phi_2\bar{\phi}_2},\\
	
	h_{011}^{2k_1}(\theta)=&0,\\
	
	 h_{020}^0(\theta)=&\frac{1}{2}[2i\omega_0I-\int_{-r}^0e^{2i\omega_0\theta}\mathrm{d}\eta _{0}(\theta)]^{-1}Q_{\phi_2\phi_2}e^{2i\omega_0\theta}-\frac{1}{2i\omega_0}[\phi_2(\theta)\psi_2(0)+\frac{1}{3}\bar{\phi}_2(\theta)\bar{\psi}_2(0)]Q_{\phi_2\phi_2},\\
	
	 h_{110}^{k_1}(\theta)=&[i\omega_0I-\int_{-r}^0e^{i\omega_0\theta}\mathrm{d}\eta _{k_1}(\theta)]^{-1}Q_{\phi_1\phi_2}e^{i\omega_0\theta}-\frac{1}{i\omega_0}\phi_1(0)\psi_1(0)Q_{\phi_1\phi_2},\\
	
	 h_{002}^0(\theta)=&\overline{h_{020}^0(\theta)},~~~h_{101}^{k_1}(\theta)=\overline{h_{110}^{k_1}(\theta)}.
	\end{array}
	\end{equation}
	where $\theta\in[-r,0]$.
	Then we have the following result.
	
\begin{pro}\label{pro:4.3} For $k_2=0$, $k_1\neq 0$ and Neumann boundary condition on spatial domain $\Omega =(0,l\pi),~l>0$, the parameters $a_{11},~a_{23},~a_{111},~a_{123},~b_{12},~b_{112},~b_{223}$ in \eqref{third} are
	\begin{equation}\label{eq03-3++}
	\begin{array}{rlc}
	a_{11}=&a_{23}=b_{12}=0,\\
	
	a_{111}=&\frac{1}{4}\psi_1 (0)C_{\phi_1\phi_1\phi_1}+\frac{1}{\omega_0}\psi_1 (0)\mathrm{Re}(\mathrm{i}Q_{\phi_1\phi_2}\psi_2 (0))Q_{\phi_1\phi_1}+\psi_1 (0) Q_{\phi_1} ( h_{200}^{0}+\frac{1}{\sqrt{2}}h_{200}^{2k_1}),\\
	
	a_{123}=&\psi_1 (0)C_{\phi_1\phi_2\bar{\phi}_2}+\frac{2}{\omega_0}\psi_1 (0)\mathrm{Re}(\mathrm{i}Q_{\phi_1\phi_2}\psi_2 (0))Q_{\phi_2\bar{\phi}_2}+\\
	&\psi_1 (0) [Q_{\phi_1} ( h_{011}^{0}+\frac{1}{\sqrt{2}}h_{011}^{2k_1})
	+Q_{\phi_2}  h_{101}^{k_1}
	+Q_{\bar{\phi}_2}  h_{110}^{k_1}],\\
	
	b_{112}=&\frac{1}{2}\psi_2 (0)C_{\phi_1\phi_1\phi_2}+\frac{1}{2i\omega_0}\psi_2 (0)\{2Q_{\phi_1\phi_1}\psi_1 (0) Q_{\phi_1\phi_2} +[-Q_{\phi_2\phi_2}\psi_2 (0)\\

    &+Q_{\phi_2\bar{\phi}_2}\bar{\psi}_2 (0)]Q_{\phi_1\phi_1} \}+\psi_2 (0) (Q_{\phi_1}  h_{110}^{k_1}+Q_{\phi_2}  h_{200}^{0}),\\
	
	b_{223}=&\frac{1}{2}\psi_2 (0)C_{\phi_2\phi_2\bar{\phi}_2}+\frac{1}{4i\omega_0}\psi_2 (0)\{
	\frac{2}{3}Q_{\bar{\phi}_2\bar{\phi}_2}\bar{\psi}_2 (0)Q_{\phi_2\phi_2}+[-2Q_{\phi_2\phi_2}\psi_2 (0)\\&
	
	+4Q_{\phi_2\bar{\phi}_2}\bar{\psi}_2 (0)]Q_{\phi_2\bar{\phi}_2}\}+\psi_2 (0) (Q_{\phi_2} h_{011}^0+Q_{\bar{\phi}_2} h_{020}^0).
	\end{array}
	\end{equation}
	\end{pro}
	
	%%%\frac{d}{dt}
	%%%
	\noindent{\bf Case (4) $k_1=0$, $k_2\neq 0$.}
	
Here we have
	$$
	\begin{array}{rl}
	\langle\beta_{k_1}^2,\beta_{k_1}\rangle=& \langle\beta_{k_2}^2,\beta_{k_1} \rangle=\langle\beta_{k_1}\beta_{k_2},\beta_{k_2} \rangle=1,\\
	\langle\beta_{k_1}^2,\beta_{k_2}\rangle=& \langle\beta_{k_2}^2,\beta_{k_2} \rangle=\langle\beta_{k_1}\beta_{k_2},\beta_{k_1} \rangle=0,\\
	\langle\beta_{k_1}^3,\beta_{k_1}\rangle=& 1,~\langle\beta_{k_2}^3,\beta_{k_2} \rangle=\frac{3}{2},~\langle\beta_{k_1}\beta_{k_2}^2,\beta_{k_1} \rangle=\langle\beta_{k_1}^2\beta_{k_2},\beta_{k_2} \rangle=1,\\
	
	\langle Q_{\phi_1 h_{200}}\beta_{k_1},\beta_{k_1}\rangle=&Q_{\phi_1} h_{200}^{0},~~~~~\langle Q_{\phi_1 h_{011}}\beta_{k_1},\beta_{k_1}\rangle=Q_{\phi_1}  h_{011}^{0},\\
	
	\langle Q_{\phi_2 h_{101}}\beta_{k_2},\beta_{k_1}\rangle=&Q_{\phi_2}  h_{101}^{k_2},~~~~~\langle Q_{\bar{\phi}_2 h_{110}}\beta_{k_2},\beta_{k_1}\rangle=Q_{\bar{\phi}_2}  h_{110}^{k_2},\\
	
	\langle Q_{\phi_1 h_{101}}\beta_{k_1},\beta_{k_2}\rangle=&Q_{\phi_1}  h_{101}^{k_2},~~~~~\langle Q_{\phi_1 h_{110}}\beta_{k_1},\beta_{k_2}\rangle=Q_{\phi_1}  h_{110}^{k_2},\\
	
	\langle Q_{\phi_2 h_{200}}\beta_{k_2},\beta_{k_2}\rangle=&Q_{\phi_2} ( h_{200}^{0}+\frac{1}{\sqrt{2}}h_{200}^{2k_2}),~~~\langle Q_{\phi_2 h_{011}}\beta_{k_2},\beta_{k_2}\rangle=Q_{\phi_2} ( h_{011}^{0}+\frac{1}{\sqrt{2}}h_{011}^{2k_2}),\\
	
	\langle Q_{\bar{\phi}_2 h_{200}}\beta_{k_2},\beta_{k_2}\rangle=&Q_{\bar{\phi}_2} ( h_{200}^{0}+\frac{1}{\sqrt{2}}h_{200}^{2k_2}),~~~\langle Q_{\bar{\phi}_2 h_{011}}\beta_{k_2},\beta_{k_2}\rangle=Q_{\bar{\phi}_2} ( h_{011}^{0}+\frac{1}{\sqrt{2}}h_{011}^{2k_2}),\\
	
	\langle Q_{\bar{\phi}_2 h_{020}}\beta_{k_2},\beta_{k_2}\rangle=&Q_{\bar{\phi}_2} ( h_{020}^{0}+\frac{1}{\sqrt{2}}h_{020}^{2k_2}),~~~\langle Q_{\phi_2 h_{002}}\beta_{k_2},\beta_{k_2}\rangle=Q_{\phi_2}  ( h_{002}^{0}+\frac{1}{\sqrt{2}}h_{002}^{2k_2}),
	\end{array}
	$$
	%By (\ref{eq16+2}) and(\ref{eq16+3-}) the obtained results about $ h_{200}^{0}, h_{200}^{2k_2},h_{011}^{0},h_{011}^{2k_2},h_{101}^{k_2},h_{110}^{k_2},h_{020}^{0},h_{020}^{2k_2},h_{002}^{0},h_{002}^{2k_2}$, which  need to be obtained  in (\ref{eq03-5+++}), are as follows.
and	%%%%
	\begin{equation}\label{eq03-4}
	\begin{array}{rlc}
	 h_{200}^0(\theta)=&\frac{1}{2}\theta\phi_1(0)\psi_1(0)Q_{\phi_1\phi_1}+E_{200},\;\;
	h_{200}^{2k_2}(\theta)\equiv 0,\\
	
	 h_{011}^0(\theta)=&\theta\phi_1(0)\psi_1(0)Q_{\phi_2\bar{\phi}_2}+E_{011},\;\;
	h_{011}^{2k_2}(\theta)\equiv -\frac{1}{\sqrt{2}}[\int_{-r}^0\mathrm{d}\eta _{2k_2}(\theta)]^{-1}Q_{\phi_2\bar{\phi}_2},\\
	
	 h_{020}^0(\theta)=&-\frac{1}{4i\omega_0}\phi_1(0)\psi_1(0)Q_{\phi_2\phi_2}+E_{020}e^{2i\omega_0\theta},\\
	h_{020}^{2k_2}(\theta)= &\frac{1}{2\sqrt{2}}[2i\omega_0I-\int_{-r}^0e^{2i\omega_0\theta}\mathrm{d}\eta _{2k_2}(\theta)]^{-1}Q_{\phi_2\phi_2}e^{2i\omega_0\theta},\\
	
	 h_{110}^{k_2}(\theta)=&\frac{1}{i\omega_0}[i\omega_0\theta\phi_2(\theta)\psi_2(0)-\frac{1}{2}\bar{\phi}_2(\theta)\bar{\psi}_2(0))]Q_{\phi_1\phi_2}+E_{110}e^{i\omega_0\theta},\\
	
	h_{002}^0(\theta)=&\overline{h_{020}^0(\theta)},~~~
	h_{002}^{2k_2}(\theta)= \overline{h_{020}^{2k_2}(\theta)},~~~h_{101}^{k_2}(\theta)=\overline{h_{110}^{k_2}(\theta)},
	\end{array}
	\end{equation}
	where $\theta\in[-r,0]$ and the constant vectors $E_q^0$ for $q\in \mathbb{N}_0^3$ with $|q|=2$, satisfy the following equations
	\begin{equation}\label{eq03-4+}
	\begin{array}{rlc}
	 [\int_{-r}^0\mathrm{d}\eta_0(\theta)]E_{200}=&\frac{1}{2}[-I+(I-\int_{-r}^0\theta \mathrm{d}\eta_0(\theta))\phi_1(0)\psi_1(0)]Q_{\phi_1\phi_1},\\
	
	[\int_{-r}^0 \mathrm{d} \eta_0(\theta)]E_{011}=&[-I+(I-\int_{-r}^0\theta \mathrm{d} \eta_0(\theta))\phi_1(0)\psi_1(0)]Q_{\phi_2\bar{\phi}_2},\\
	
	E_{020}=&\frac{1}{2}[2i\omega_0I-\int_{-r}^0e^{2i\omega_0\theta} \mathrm{d} \eta_0(\theta)]^{-1}Q_{\phi_2\phi_2},\\
	
	 [i\omega_0I-\int_{-r}^0e^{i\omega_0\theta}\mathrm{d}\eta_{k_2}(\theta)]E_{110}=&[I-\phi_2(0)\psi_2(0)+\int_{-r}^0\theta \mathrm{d}\eta_{k_2}(\theta)\phi_2(\theta)\psi_2(0)]Q_{\phi_1\phi_2}.
	\end{array}
	\end{equation}
	Then we have the following result.
	
\begin{pro}\label{pro:4.4} For $k_1=0$, $k_2\neq 0$ and Neumann boundary condition on spatial domain $\Omega =(0,l\pi),~l>0$, the parameters $a_{11},~a_{23},~a_{111},~a_{123},~b_{12},~b_{112},~b_{223}$ in \eqref{third} are
	\begin{equation}\label{eq03-4++}
	\begin{array}{rlc}
	a_{11}=&\frac{1}{2}\psi_1 (0)Q_{\phi_1\phi_1},\;\;
	a_{23}=\psi_1 (0)Q_{\bar{\phi}_2\phi_2},\;\;
	b_{12}=\psi_2 (0)Q_{\phi_1\phi_2},\\
	
	a_{111}=&\frac{1}{6}\psi_1 (0)C_{\phi_1\phi_1\phi_1}+\psi_1 (0) Q_{\phi_1} h_{200}^0,\\
	
	a_{123}=&\psi_1 (0)C_{\phi_1\phi_2\bar{\phi}_2}+\frac{1}{\omega_0}\psi_1 (0)\mathrm{Re}(\mathrm{i}Q_{\phi_2\phi_2}\psi_2 (0)Q_{\phi_1\bar{\phi}_2})+\psi_1 (0) (Q_{\phi_1} h_{011}^0+Q_{\phi_2} h_{101}^{k_2}+Q_{\bar{\phi}_2} h_{110}^{k_2}),\\
	
	b_{112}=&\frac{1}{2}\psi_2 (0)C_{\phi_1\phi_1\phi_2}+\frac{1}{2i\omega_0}\psi_2 (0)Q_{\phi_1\bar{\phi}_2}\bar{\psi}_2 (0)Q_{\phi_1\phi_2}
	+\psi_2 (0) [Q_{\phi_1} h_{110}^{k_2}+Q_{\phi_2} ( h_{200}^{0}+\frac{1}{\sqrt{2}}h_{200}^{2k_2})],\\
	
	b_{223}=&\frac{3}{4}\psi_2 (0)C_{\phi_2\phi_2\bar{\phi}_2}+\frac{1}{4i\omega_0}\psi_2 (0)Q_{\phi_1\bar{\phi}_2}\psi_1 (0)Q_{\phi_2\phi_2} +\psi_2 (0) [Q_{\phi_2} ( h_{011}^{0}+\frac{1}{\sqrt{2}}h_{011}^{2k_2})\\
	&+Q_{\bar{\phi}_2} ( h_{020}^{0}+\frac{1}{\sqrt{2}}h_{020}^{2k_2})].
	\end{array}
	\end{equation}
	\end{pro}
	%%%
	%%%
	\noindent{\bf Case (5) $k_1\neq k_2$, $k_1,k_2\neq0$.}
	
Here we have
	$$
	\begin{array}{rl}
	\langle\beta_{k_1}^2,\beta_{k_1}\rangle=&0,~ \langle\beta_{k_2}^2,\beta_{k_1} \rangle=\langle\beta_{k_1}\beta_{k_2},\beta_{k_2} \rangle=\frac{1}{\sqrt{2}}\delta(k_1-2k_2),\\
	
	\langle\beta_{k_2}^2,\beta_{k_2} \rangle=&0,~\langle\beta_{k_1}^2,\beta_{k_2}\rangle=\langle\beta_{k_1}\beta_{k_2},\beta_{k_1} \rangle=\frac{1}{\sqrt{2}}\delta(k_2-2k_1),\\
	
	\langle\beta_{k_1}^3,\beta_{k_1}\rangle=&\langle\beta_{k_2}^3,\beta_{k_2} \rangle=\frac{3}{2},~\langle\beta_{k_1}\beta_{k_2}^2,\beta_{k_1} \rangle=\langle\beta_{k_1}^2\beta_{k_2},\beta_{k_2} \rangle=1,\\
	
	\langle Q_{\phi_1 h_{200}}\beta_{k_1},\beta_{k_1}\rangle=&\frac{1}{\sqrt{2}}Q_{\phi_1}h_{200}^{2k_1}+Q_{\phi_1}h_{200}^{0},\;\;

\langle Q_{\phi_1 h_{011}}\beta_{k_1},\beta_{k_1}\rangle=\frac{1}{\sqrt{2}}Q_{\phi_1}h_{011}^{2k_1}+Q_{\phi_1}h_{011}^{0},\\
	
	\langle Q_{\phi_2 h_{101}}\beta_{k_2},\beta_{k_1}\rangle=&Q_{\phi_2} (\frac{1}{\sqrt{2}} h_{101}^{|k_1-k_2|}+\frac{1}{\sqrt{2}} h_{101}^{k_1+k_2}+h_{101}^{0}),\\
	
	\langle Q_{\bar{\phi}_2 h_{110}}\beta_{k_2},\beta_{k_1}\rangle=&Q_{\bar{\phi}_2}  (\frac{1}{\sqrt{2}} h_{110}^{|k_1-k_2|}+\frac{1}{\sqrt{2}} h_{110}^{k_1+k_2}+h_{110}^{0}),\\
	
	\langle Q_{\phi_1 h_{101}}\beta_{k_1},\beta_{k_2}\rangle=&Q_{\phi_1} (\frac{1}{\sqrt{2}} h_{101}^{|k_1-k_2|}+\frac{1}{\sqrt{2}} h_{101}^{k_1+k_2}+h_{101}^{0}),\\
	
	\langle Q_{\phi_1 h_{110}}\beta_{k_1},\beta_{k_2}\rangle=&Q_{\phi_1}  (\frac{1}{\sqrt{2}} h_{110}^{|k_1-k_2|}+\frac{1}{\sqrt{2}} h_{110}^{k_1+k_2}+h_{110}^{0}),\\
	
	\langle Q_{\phi_2 h_{200}}\beta_{k_2},\beta_{k_2}\rangle=&\frac{1}{\sqrt{2}}Q_{\phi_2}  h_{200}^{2k_2}+Q_{\phi_2}  h_{200}^{0},\;\;

\langle Q_{\phi_2 h_{011}}\beta_{k_2},\beta_{k_2}\rangle=\frac{1}{\sqrt{2}}Q_{\phi_2}  h_{011}^{2k_2}+Q_{\phi_2}  h_{011}^{0},\\
	
	\langle Q_{\bar{\phi}_2 h_{200}}\beta_{k_2},\beta_{k_2}\rangle=&\frac{1}{\sqrt{2}}Q_{\bar{\phi}_2}  h_{200}^{2k_2}+Q_{\bar{\phi}_2}  h_{200}^{0},\;\;

\langle Q_{\bar{\phi}_2 h_{011}}\beta_{k_2},\beta_{k_2}\rangle=\frac{1}{\sqrt{2}}Q_{\bar{\phi}_2}  h_{011}^{2k_2}+Q_{\bar{\phi}_2}  h_{011}^{0},\\
	
	\langle Q_{\bar{\phi}_2 h_{020}}\beta_{k_2},\beta_{k_2}\rangle=&\frac{1}{\sqrt{2}}Q_{\bar{\phi}_2}  h_{020}^{2k_2}+Q_{\bar{\phi}_2}  h_{020}^{0},\;\;

\langle Q_{\phi_2 h_{002}}\beta_{k_2},\beta_{k_2}\rangle=\frac{1}{\sqrt{2}}Q_{\phi_2}  h_{002}^{2k_2}+Q_{\phi_2}  h_{002}^{0},
	\end{array}
	$$
	where $\delta(x)=1,~\mathrm{for}~x=0~\mathrm{and}~\delta(x)=0,~\mathrm{for}~x\neq0 $. % The function By (\ref{eq16+2}) and(\ref{eq16+3-}) the obtained results about
And $ h_{200}^{2k_1},~ h_{200}^{2k_2},~h_{011}^{2k_1},~h_{011}^{2k_2},~h_{101}^{0},~h_{101}^{|k_1-k_2|}$, $h_{101}^{k_1+k_2},h_{110}^{0},~h_{110}^{|k_1-k_2|},~h_{110}^{k_1+k_2},~h_{020}^{2k_2},~h_{002}^{2k_2}$ are given by%, which  need to be obtained  in (\ref{eq03-5+++}), are as follows.
	\begin{equation}\label{eq03-5}
	\begin{array}{rlc}
	h_{200}^0(\theta)=&-\frac{1}{2}[\int_{-r}^0\mathrm{d}\eta _0(\theta)]^{-1}Q_{\phi_1\phi_1}+\frac{1}{2i\omega_0}(\phi_2(\theta)\psi_2(0)-\bar{\phi}_2(\theta)\bar{\psi}_2(0))]Q_{\phi_1\phi_1},\\
	
	h_{200}^{2k_1}(\theta)\equiv &-\frac{1}{2\sqrt{2}}[\int_{-r}^0\mathrm{d}\eta _{2k_1}(\theta)]^{-1}Q_{\phi_1\phi_1},\\
	h_{200}^{2k_2}(\theta)\equiv &0,~~~h_{011}^{2k_1}\equiv 0,\\
	h_{011}^{2k_2}(\theta)=&-\frac{1}{\sqrt{2}}[\int_{-r}^0\mathrm{d}\eta _{2k_2}(\theta)]^{-1}Q_{\phi_2\bar{\phi}_2},\\

	 h_{020}^{2k_2}(\theta)=&\frac{1}{2\sqrt{2}}[2i\omega_0I-\int_{-r}^0e^{2i\omega_0\theta}\mathrm{d}\eta _{2k_2}(\theta)]^{-1}Q_{\phi_2\phi_2}e^{2i\omega_0\theta},\\
	
	 h_{110}^{|k_1-k_2|}(\theta)=&\frac{1}{\sqrt{2}}[i\omega_0I-\int_{-r}^0e^{i\omega_0\theta}\mathrm{d}\eta _{|k_1-k_2|}(\theta)]^{-1}Q_{\phi_1\phi_2}e^{i\omega_0\theta},\\
	
	 h_{110}^{k_1+k_2}(\theta)=&\frac{1}{\sqrt{2}}[i\omega_0I-\int_{-r}^0e^{i\omega_0\theta}\mathrm{d}\eta _{k_1+k_2}(\theta)]^{-1}Q_{\phi_1\phi_2}e^{i\omega_0\theta},\\
	
	h_{110}^{0}(\theta)\equiv &0,~~~
	h_{101}^{0}(\theta)\equiv 0,\\
	
	 h_{002}^{2k_2}(\theta)=&\overline{h_{020}^{2k_2}(\theta)},~~~h_{101}^{|k_1-k_2|}(\theta)=\overline{h_{110}^{|k_1-k_2|}(\theta)},~~~
	
	h_{101}^{k_1+k_2}(\theta)=\overline{h_{110}^{k_1+k_2}(\theta)}.
	\end{array}
	\end{equation}
	where $\theta\in[-r,0]$.
	Thus we have the following result.

\begin{pro}\label{pro:4.5} For $k_1\neq k_2$, $k_1,k_2\neq0$ and Neumann boundary condition on spatial domain $\Omega =(0,l\pi),~l>0$, the parameters $a_{11},~a_{23},~a_{111},~a_{123},~b_{12},~b_{112},~b_{223}$ in \eqref{third} are
	\begin{equation}\label{eq03-5++}
	\begin{array}{rlc}
	a_{11}=&0,\;\;
	a_{23}=\frac{1}{\sqrt{2}}\delta(k_1-2k_2)\psi_1 (0)Q_{\bar{\phi}_2\phi_2},\;\;
	b_{12}=\frac{1}{\sqrt{2}}\delta(k_1-2k_2)\psi_2 (0)Q_{\phi_1\phi_2},\\
	
	a_{111}=&\frac{1}{4}\psi_1 (0)C_{\phi_1\phi_1\phi_1}+\frac{1}{2\omega_0}\delta(k_2-2k_1)\psi_1 (0)\mathrm{Re}(\mathrm{i}Q_{\phi_1\phi_2}\psi_2 (0))Q_{\phi_1\phi_1}+\frac{1}{\sqrt{2}}\psi_1 (0) Q_{\phi_1}h_{200}^{2k_1},\\
	
	a_{123}=&\psi_1 (0)C_{\phi_1\phi_2\bar{\phi}_2}+\frac{1}{2\omega_0}\psi_1 (0)\delta(k_1-2k_2)\mathrm{Re}(\mathrm{i}Q_{\phi_2\phi_2}\psi_2 (0)Q_{\phi_1\bar{\phi}_2})+\psi_1 (0) [\frac{1}{\sqrt{2}}Q_{\phi_1} h_{011}^{2k_1}+\\
	
	&Q_{\phi_2} (\frac{1}{\sqrt{2}} h_{101}^{|k_1-k_2|}+\frac{1}{\sqrt{2}} h_{101}^{k_1+k_2}+h_{101}^{0})+Q_{\bar{\phi}_2}  (\frac{1}{\sqrt{2}} h_{110}^{|k_1-k_2|}+\frac{1}{\sqrt{2}} h_{110}^{k_1+k_2}+h_{110}^{0})],\\
	
	b_{112}=&\frac{1}{2}\psi_2 (0)C_{\phi_1\phi_1\phi_2}+\frac{1}{4i\omega_0}\psi_2 (0)[2\delta(k_2-2k_1)Q_{\phi_1\phi_1}\psi_1 (0) +\delta(k_1-2k_2)Q_{\phi_1\bar{\phi}_2}\bar{\psi}_2 (0)]Q_{\phi_1\phi_2} \\&
	
	+\psi_2 (0) [Q_{\phi_1} (\frac{1}{\sqrt{2}} h_{110}^{|k_1-k_2|}+\frac{1}{\sqrt{2}} h_{110}^{k_1+k_2}+h_{110}^{0})+\frac{1}{\sqrt{2}}Q_{\phi_2}  h_{200}^{2k_2}],\\
	
	b_{223}=&\frac{3}{4}\psi_2 (0)C_{\phi_2\phi_2\bar{\phi}_2}+\frac{1}{8i\omega_0}\delta(k_1-2k_2)\psi_2 (0)Q_{\phi_1\bar{\phi}_2}\psi_1 (0)Q_{\phi_2\phi_2} \\&
+\psi_2 (0) [Q_{\phi_2}  (h_{011}^{0}+ \frac{1}{\sqrt{2}}h_{011}^{2k_2})+Q_{\bar{\phi}_2}  (h_{020}^{0}+\frac{1}{\sqrt{2}}h_{020}^{2k_2})].
	\end{array}
	\end{equation}
	\end{pro}
	\section{Example}
	In this section we  apply our results above to the Turing-Hopf bifurcation of a diffusive Schnakenberg chemical reaction system with gene expression time delay in the following form (see \cite{Chen2013JNS,YGLM,SER}):
	
\begin{equation}\label{3.1b}
\begin{cases}
u_t(x,t)=\varepsilon d u_{xx}(x,t)+a-u(x,t)+u^2(x,t-\tau)v(x,t-\tau), &x\in(0,1),\;t>0,\\
v_t(x,t)=d v_{xx}(x,t)+b-u^2(x,t-\tau)v(x,t-\tau), &x\in(0,1),\;t>0,\\
u_x(0,t)=u_x(1,t)=v_x(0,t)=v_x(1,t)=0, \;\; &t\geq0,\\
u(x,t)=\phi(x,t)\geq0, v(x,t)=\varphi(x,t)\geq0,&(x,t)\in[0,1]\times[-\tau,0],
\end{cases}
\end{equation}
System \eqref{3.1b} has a unique positive constant steady state
solution $E_*=(u_*,v_*)$, where
\begin{equation}\label{fixp}
u_*=a+b,\;~~~~v_*=\frac{b}{(a+b)^2},
\end{equation}

Recalling that $\mu_k=k^2\pi^2$ are the eigenvalues of the $-\Delta$ in the one dimensional spatial domain $(0,1)$, $k\in\mathbb{N}_0$. Then, a straightforward analysis shows that the eigenvalues of the linearized operator are given by the roots of
\begin{equation}\label{eigen}
D_k(\lambda):=\lambda^2+p_k\lambda+r_k+(s_k\lambda+q_k)e^{-\lambda\tau}=0,\;\;k\in\mathbb{N}_0,
\end{equation}
where,
\begin{equation}\label{cf}
\begin{split}
p_k=& (\varepsilon+1)dk^2\pi^2+1,\;~~~r_k= \varepsilon d^2k^4\pi^4+dk^2\pi^2,\\
s_k=& u_*^2-2u_*v_*,\;~~~~~~~~~~~q_k= (\varepsilon u_*^2-2u_*v_*)dk^2\pi^2+u_*^2.
\end{split}
\end{equation}
 By analyzing the characteristic equations \eqref{eigen} with $a=1,\ b=2,\ d=4$ (see \cite [Theorem 2.12 and 2.15] {JWC2017} for details on general results),  we have
\begin{theorem}\label{thm:5.1}
For system \eqref{3.1b} with $a=1$,  $b=2$, $d=4$, there is a constant positive steady state $(u_*,v_*)=(3,2/9)$, and there exists  $\tau_*\approx 0.2014$, $\varepsilon_*\approx 0.0022$, $ \omega_*\approx 7.6907$ such that
\begin{enumerate}
	\item when $\tau=\tau_*$, $\varepsilon=\varepsilon_*$,  $D_0(\lambda)$ has a pair of purely imaginary roots $\pm \mathrm{i}\omega_*$, $D_1(\lambda)$ has a simple zero root, with all other roots  of $D_{k}(\lambda)$ having negative real parts $k\in\mathbb{N}_0$.
	\item the system \eqref {3.1b} undergoes $(1,0)-$mode Turing-Hopf bifurcation near the constant steady state $(u_*,v_*)$  at $\tau=\tau_*,\ \varepsilon=\varepsilon_*$.% with $0\leq k<K^*$ and $j, k\in \mathbb{N}_0$.
	\item the constant steady state $(u_*,v_*)$ is locally  asymptotically stable for the system \eqref{3.1b} with $\tau\in [\,0,\tau_*)$ and $\varepsilon>\varepsilon_*$, and unstable for $0<\varepsilon<\varepsilon_*$ or $\tau > \tau_*$.
\end{enumerate}
\end{theorem}

Hence we have $k_1=1$ and $k_2=0$ at $\tau=\tau_*$ defined in Theorem \ref{thm:5.1}, which corresponds to \textbf{Case (3)} in Section 4.
We normalize the time delay $\tau$ in system \eqref{3.1b} by time-rescaling $t\rightarrow t/\tau$, and translate $(u_*,v_*)$ into the origin. We also introduce two bifurcation parameters
$\alpha=(\alpha_1,\alpha_2)$ by setting
\begin{equation}
\tau=\tau_*+\alpha_1,\;\; \varepsilon=\varepsilon_*+\alpha_2.
\end{equation}
Then, system \eqref{3.1b} is transformed into an abstract equation in $C([-1,0],X)$:
%	Then, by (\ref{eq6+-}) and (\ref{eq6+}) system (\ref{rdu2}) can be written in the following form in
%	as,
	\begin{equation}\label{lin}
	\frac{d}{dt}U(t)=L_0 U_t+D_0\Delta U(t)+\frac{1}{2}L_1(\alpha)U_t+\frac{1}{2}D_1(\alpha)\Delta U(t)+\frac{1}{2!}Q(U_t,U_t)+\frac{1}{3!}C(U_t,U_t,U_t)+\ldots,
	\end{equation}
	where
\begin{align*}
	D_0&=d\tau_* \left(
	\begin{array}{cc}
	\varepsilon_*&0\\
	0&1
	\end{array}\right),\;\;
	D_1(\alpha)=2d\left(
	\begin{array}{cc}
	\alpha_1\varepsilon_*+\alpha_2\tau_*&0\\
	0&\alpha_1
	\end{array}\right),\\
 L_0 X&=\tau_*\left(
	\begin{array}{c}
	-x_1(0)+2u_*v_*x_1(-1)+u_*^2x_2(-1) \\
	-(2u_*v_*x_1(-1)+u_*^2x_2(-1))
	\end{array}\right),\\
    L_1(\alpha)X&=2\alpha_1\left(
	\begin{array}{c}
	-x_1(0)+2u_*v_*x_1(-1)+u_*^2x_2(-1) \\
	-(2u_*v_*x_1(-1)+u_*^2x_2(-1))
	\end{array}\right),
\end{align*}
and
\begin{align*}
Q_{XY}&=2\tau_*[v_*x_1(-1)y_1(-1)+u_*(x_1(-1)y_2(-1)+x_2(-1)y_1(-1))]\left(
	\begin{array}{c}
	1\\
	-1
	\end{array}\right),\\
C_{XYZ}&=2\tau_*[x_1(-1)y_1(-1)z_2(-1)+x_1(-1)y_2(-1)z_1(-1)+x_2(-1)y_1(-1)z_1(-1)]\left(
	\begin{array}{c}
	1\\
	-1
	\end{array}\right),
\end{align*}
	with $
	X=\left(
	\begin{array}{c}
	x_1\\
	x_2
	\end{array}\right), ~Y=\left(
	\begin{array}{c}
	y_1\\
	y_2
	\end{array}\right),~Z=\left(
	\begin{array}{c}
	z_1\\
	z_2
	\end{array}\right).$
	
From routine calculation, we obtain the eigenfunctions (as defined in \eqref{eq5+3}):
\begin{equation}\label{eq88}
\begin{split}
\phi_1(0) &=\left(
	\begin{array}{c}
	1 \\
	-0.0274
	\end{array}\right),~~
	\phi_2(0) =\left(
	\begin{array}{c}
	1 \\
	-1 + 0.1298\mathrm{i}
	\end{array}\right),%e^{\texttt{i}\tau_*\omega_*\theta},~~-1\leq \theta \leq 0,\\
	\\ \psi_1(0 ) &=\frac{1}{1.1734}\left(
	1 ,
	0.1849
	\right),~~
	\psi_2(0) =\frac{1}{-8.1518 - 6.9779\mathrm{i}}\left(
	1 ,
	6.7502-0.8761\mathrm{i}
	\right).
\end{split}
 \end{equation}

By \eqref{eq03-3}, we obtain that
\begin{equation}\label{5.8}
\begin{split}
h_{200}^0(0)&=\left(\begin{array}{c}-0.0062\\0.0004
\end{array}\right),\;
h_{200}^0(-1)=\left(\begin{array}{c}-0.0055\\-0.0018
\end{array}\right),\;
h_{200}^2(0)= h_{200}^2(-1)=\left(\begin{array}{c}0.4506\\-0.0038
\end{array}
\right)
\\
h_{011}^0(0)&=\left(\begin{array}{c}1.2336\\-0.0877\end{array}\right),
h_{011}^0(-1)=\left(\begin{array}{c}1.0906\\0.3504\end{array}\right),
h_{011}^2(0)=\left(\begin{array}{c}0\\0\end{array}\right),
h_{011}^2(-1)=\left(\begin{array}{c}0\\0\end{array}\right),
\end{split}
\end{equation}
\begin{equation}\label{5.9}
\begin{split}
h_{020}^0(0)&=\left(\begin{array}{c}0.0761+0.0358\mathrm{i}\\-0.0748 +0.0093\mathrm{i}\end{array}\right),\;
h_{020}^0(-1)=\left(\begin{array}{c}0.2954-0.1131\mathrm{i}\\-0.2848 + 0.1679\mathrm{i}\end{array}\right),\\
h_{110}^1(0)&=\left(\begin{array}{c}0.1171 +0.1850\mathrm{i}\\-0.0029- 0.1255\mathrm{i}\end{array}\right),\;
h_{110}^1(-1)=\left(\begin{array}{c}-0.3783- 0.5733\mathrm{i}\\-0.1100 + 0.0128\mathrm{i}\end{array}\right),\\
h_{101}^1&=\overline{h_{110}^1}, \;\;\; h_{002}^0=\overline{h_{020}^0}.
\end{split}
\end{equation}

Substituting the above calculated values into the expression \eqref {eq03-3++}, we obtain the coefficients of normal form \eqref{third} as follows
\begin{equation}\label{z123}
\begin{aligned}
&a_1(\alpha)=-0.0009\alpha_1-6.7762\alpha_2,\;\;
b_2(\alpha)=(3.5818 + 2.2515 \mathrm{i})\alpha_1,\\
&a_{11}=a_{23}=b_{12}=0,\\
&a_{111}= -9.4377\times 10^{-4},~~~~~~~b_{112}=0.0403 + 0.1213 \mathrm{i}, \\
&a_{123}=-0.4782,~~~~~~~~~~~~~~~~~b_{223}= -0.2553 -0.7712\mathrm{i}.
\end{aligned}
\end{equation}
Thus, in the corresponding  planar system \eqref{eq473-2}, we have that
\begin{equation}
\begin{split}
 \varepsilon_1(\alpha)&=3.5818 \alpha_1, \;\; \varepsilon_2(\alpha)=-0.0009\alpha_1-6.7762\alpha_2,\\
 b_0&=-42.7011,\; c_0=1.8735,\; d_0=1, \; \mathrm{sign}(\mathrm{Re}(b_{223}))=-1.
 \end{split}
\end{equation}
Therefore the Case $\mathrm{III}$ in Table \ref{tab1} occurs, and we find that the  bifurcation critical lines in Figure \ref{fig1} are, respectively,
\begin{equation*}
\begin{split}
L_1:&\tau=\tau_*,~\varepsilon>\varepsilon_*,\;\;
L_2:\varepsilon=\varepsilon_*-0.00013(\tau-\tau_*),~\tau>\tau_*,\\
L_3:&\varepsilon=\varepsilon_*-0.9916(\tau-\tau_*),~\tau>\tau_*,\;\;
L_4:\tau=\tau_*,~\varepsilon<\varepsilon_*,\\
L_5:&\varepsilon=\varepsilon_*+ 0.0111(\tau-\tau_*),~\tau<\tau_*,\;\;
L_6:\varepsilon=\varepsilon_*-0.00013(\tau-\tau_*),~\tau<\tau_*.
\end{split}
\end{equation*}

 Taking notice of $\mathrm{sign}(\mathrm{Re}(b_{223}))=-1$ in the coordinate transformation \eqref{tran-2}, and from phase portraits  in Figure \ref{fig1} and the analysis in \cite[Section 4]{AJ}, we have the following result.
\begin{theorem}\label{thm:5.2}
Let $a=1$, $b=2$ and $d=4$. At the constant positive steady state $(u_*,v_*)=(3,2/9)$, near the $(1,0)$-mode Turing-Hopf bifurcation point $(\tau_*,\varepsilon_*)\approx (0.2014,0.0022)$, with frequency $\omega_*=7.6907$, the system \eqref{3.1b} has the following dynamical behavior when the parameter pair $(\tau,\varepsilon)$ is sufficiently close to  $(\tau_*,\varepsilon_*)$: (see Figure \ref{fig1})
\begin{enumerate}
\item[(1)] When $\varepsilon>\varepsilon_*-0.0013(\tau-\tau_*)$
and $\tau<\tau_*$ (that is $(\tau, \varepsilon)\in D_1$), the constant steady state $(u_*,v_*)$ is locally asymptotically stable;
and a $0-$mode Hopf bifurcation occurs at $(u_*,v_*)$  when $(\tau,\ \varepsilon)$ crosses $L_1$ transversally.

\item[(2)] When
$\varepsilon>\varepsilon_*-0.0013(\tau-\tau_*)$
and $\tau>\tau_*$ (that is $(\tau, \varepsilon)\in D_2$), the constant steady state $(u_*,v_*)$ is unstable and there exists a locally asymptotically stable spatially homogeneous periodic orbit which  bifurcates from  $(u_*,v_*)$; and
a $1-$mode Turing bifurcation occurs at $(u_*,v_*)$  when $(\tau,\ \varepsilon)$ crosses $L_2$ transversally.

\item[(3)] When
$\varepsilon_*-0.0013(\tau-\tau_*)>\varepsilon>\varepsilon_*-0.9916(\tau-\tau_*)$
and $\tau>\tau_*$ (that is $(\tau, \varepsilon)\in D_3$), the constant steady state $(u_*,v_*)$ is unstable, there are two unstable spatially non-homogeneous steady states  which bifurcate from $(u_*,v_*)$,  and the spatially homogeneous periodic orbit is  locally asymptotically stable; and a $1-$mode Turing bifurcation occurs at the spatially homogeneous periodic orbit when $(\tau,\ \varepsilon)$ crosses $L_3$ transversally.

 \item[(4)]  When
$\varepsilon<\varepsilon_*-0.9916(\tau-\tau_*)$
and $\tau>\tau_*$ (that is $(\tau, \varepsilon)\in D_4$), the constant steady state $(u_*,v_*)$ and the two spatially non-homogeneous steady state solutions are all unstable, the spatially homogeneous periodic orbit is also unstable, and there are two locally asymptotically stable spatially non-homogeneous periodic orbits  which  bifurcate from the spatially homogeneous periodic orbit, whose linear main parts are approximately
\begin{equation}\label{5.12}
E_*+\rho\phi_{2}(0)e^{\mathrm{i}\tau_*\omega_* t}+\bar{\rho}\bar{\phi}_{2}(0)e^{\mathrm{-i}\tau_*\omega_* t}\pm h\phi_{1}(0)\cos(\pi x),
\end{equation}
where $\rho$ and $h$ are some constants; and a $0-$mode Hopf bifurcation occurs at $(u_*,v_*)$ when $(\tau,\ \varepsilon)$ crosses $L_4$ transversally.

 \item[(5)]  When
$\varepsilon<\varepsilon_*+0.0111(\tau-\tau_*)$
and $\tau<\tau_*$ (that is $(\tau, \varepsilon)\in D_5$),  the constant steady state $(u_*,v_*)$ and the two spatially non-homogeneous steady state solutions are all unstable, there is no spatially homogeneous periodic orbit (disappearing through the Hopf bifurcation on $L_4$), and two  spatially non-homogeneous periodic orbits are locally asymptotically stable; and a $0-$mode Hopf bifurcation occurs at each of two spatially non-homogeneous steady state solutions when $(\tau,\ \varepsilon)$ crosses $L_5$ transversally.

 \item[(6)] When
$\varepsilon_*-0.0013(\tau-\tau_*)>\varepsilon>\varepsilon_*+0.0111(\tau-\tau_*)$
and $\tau<\tau_*$ (that is $(\tau, \varepsilon)\in D_6$), the constant steady state $(u_*,v_*)$ is unstable,  the two  spatially non-homogeneous steady state solutions are locally asymptotically stable, and there is no  spatially non-homogeneous periodic orbits (disappearing through the Hopf bifurcations on $L_5$); and a $1-$mode Turing bifurcation occurs at $(u_*,v_*)$ when $(\tau,\ \varepsilon)$ crosses $L_6$ transversally.
 \end{enumerate}
 \end{theorem}
 %\begin{remark}
%	$k-$mode Hopf (Turing) bifurcation means that the corresponding  $k-$th characteristic equation $D_{k}(\lambda)=0$ has a pair pure imaginary roots (a zero toot).
%\end{remark}
%	A Turing-Hopf
%	bifurcation occurs at $(s,\varepsilon)=(s_0,\varepsilon_*)$, where $\varepsilon_*=0.0651$ and $k_1=1$.  The coefficients of normal form (\ref{eq443}) are
%	\begin{equation}\label{z12}
%	\begin{array}{rlc}
%	\varepsilon_1(\alpha)=&-0.6171\alpha_1-1.1172\alpha_2,\\
%	\mathrm{Re}(\varepsilon_2(\alpha))=&-0.5\alpha_1&,\\
%	a_{111}=0.2675,~~~~&b_{112}=-0.2658-0.014\mathrm{i} \\
%	a_{123}=0.1809,~~~&b_{223}=-0.0692-0.5966\mathrm{i}
%	\end{array}
%	\end{equation}

We summarize the numbers of spatialtemporal patterned solutions (steady states or periodic orbits) in each parameter region $D_i$ ($1\le i\le 6$) in Table \ref{tab4}. The Morse index of a steady state solution is defined to be the number of positive eigenvalues of associated linearized equation, and the Morse index of a periodic orbit is defined to be the number of Floquet multipliers which are greater than $1$. A steady state or a periodic orbit is locally asymptotically stable if its Morse index is $0$. Hence the stable pattern for $D_1$ is the constant steady state $(u_*,v_*)$; the stable pattern for $D_2$ and $D_3$ is the spatially homogeneous periodic orbit; a pair of spatially non-homogeneous periodic orbits are the stable patterns for $D_4$ and $D_5$; and a pair of spatially non-homogeneous steady state solutions are the stable patterns for $D_6$.

\begin{table}
\begin{center}
\begin{tabular}{|c|cccccc|}
\hline
&$D_1$ &$D_2$ & $D_3$&  $D_4$ & $D_5$ & $D_6$ \\
\hline
homogeneous steady state  & 1(0) & 1(2) & 1(3) & 1(3)& 1(1) & 1(1) \\
non-homogeneous steady state & 0    &  0   & 2(2) & 2(2) &2(2) & 2(0) \\
homogeneous periodic orbit & 0 & 1(0) & 1(0) & 1(1) & 0 &0 \\
non-homogeneous periodic orbit & 0 & 0 & 0 & 2(0) & 2(0) & 0\\
\hline
\end{tabular}
\end{center}
\caption{The number of patterned solutions of  \eqref{3.1b} in each parameter regions $D_i$ ($1\le i\le 6$). Here $j(k)$ means the number of specific patterned solutions is $j$, and the Morse index of each such patterned solution is $k$.}\label{tab4}
\end{table}

\section{Proof of Theorem \ref{g3}} \label{proof}

In this section we give the proof of the main result Theorem \ref{g3}.
From \eqref{eq6+-} and \eqref{eq6+}, we denote that, for $v\in\mathcal{C}$,
\begin{equation}
\begin{split}
F_2(v,\alpha)&=L_1(\alpha)v+D_1(\alpha)\Delta v(0)+Q(v,v),\\
F_3(v,0)&=C(v,v,v).
\end{split}
\end{equation}
By doing Taylor   expansion for the nonlinear terms in
\eqref{eq7} at $(z,y,\alpha)=(0,0,0)$, we have
\begin{equation}\label{eq8}
\begin{split}
\dot{z}  &= Bz +\dfrac{1}{2!}f^1_2(z, y,\alpha)+\dfrac{1}{3!}f^1_3(z, y,\alpha)+\cdots,  \\
\frac{\mathrm{d}}{\mathrm{d}t}  y&= A_1y+\dfrac{1}{2!}f^2_2(z,
y,\alpha)+\dfrac{1}{3!}f^2_3(z, y,\alpha)+\cdots,
\end{split}
\end{equation}
where $B=\text{diag}(0,i\omega_0,-i\omega_0)$,
$f^i_j(z,y,\alpha)\ (i=1,2)$ are the homogeneous polynomials of
degree $j$ in variables $(z,y,\alpha),~z=(z_1,z_2,\bar{z}_2)\in\mathbb{C}^3,~y\in\mathcal{Q}^1,~\alpha\in V$. For the purposes of this article, we are interested in the three terms:
\begin{equation}\label{f23}
f^{1}_2(z, y,\alpha)
\triangleq
\left(\begin{array}{cc}f^{11}_2(z, y,\alpha)\\f^{12}_2(z, y,\alpha)\\ \overline{f^{12}_2}(z, y,\alpha)
\end{array} \right),\ \
f^1_3(z, 0,0)\triangleq
\left(\begin{array}{cc}f^{11}_3(z, 0,0)\\f^{12}_3(z, 0,0)\\ \overline{f^{12}_3}(z, 0,0)
\end{array} \right),
\end{equation}
\begin{equation}\label{eq6+3}
f^2_2(z,0,0)=(X_0-\Phi\Psi(0))F_2(\phi_1z_1\beta_{k_1}+(\phi_2z_2+\bar{\phi}_2\bar{z}_2)\beta_{k_2}, 0).
\end{equation}
where
\begin{equation}\label{eq6+1}
f^{1i}_2(z, y,\alpha)=\psi_i (0)\langle F_2(\phi_1z_1\beta_{k_1}+(\phi_2z_2+\bar{\phi}_2\bar{z}_2)\beta_{k_2}+y, \alpha ),\beta_{k_i}\rangle,\;\; i=1,2,
\end{equation}
\begin{equation}\label{eq6+2}
f^{1i}_3(z, 0,0)=\psi_i (0)\langle F_3(\phi_1z_1\beta_{k_1}+(\phi_2z_2+\bar{\phi}_2\bar{z}_2)\beta_{k_2}, 0 ),\beta_{k_i}\rangle,\;\; i=1,2.
\end{equation}
Noticing that $L_1(\alpha),D_1(\alpha)$ are linear,  $Q,C$ are symmetric multilinear, and together with $\Delta\beta_{k_i}=-\mu_{k_i}\beta_{k_i}$, for $i=1,2$,
we obtain that
\begin{equation}\label{eq6+1-1}
\begin{array}{rlc}
f^{11}_2(z, y,\alpha)=&\psi_1 (0)[L_1(\alpha)\phi_1z_1-\mu_{k_1}D_1(\alpha)\phi_1(0)z_1+Q_{\phi_1\phi_1}z_1^2\langle\beta_{k_1}^2,\beta_{k_1}\rangle+2(Q_{\phi_1\phi_2}z_1z_2+\\

&Q_{\phi_1\bar{\phi}_2}z_1\bar{z}_2)\langle\beta_{k_1}\beta_{k_2},\beta_{k_1}\rangle+(Q_{\phi_2\phi_2}z_2^2+2Q_{\phi_2\bar{\phi}_2}z_2\bar{z}_2+Q_{\bar{\phi}_2\bar{\phi}_2}\bar{z}_2^2 )\langle\beta_{k_2}^2,\beta_{k_1}\rangle+\\
&\langle L_1(\alpha)y,\beta_{k_1}\rangle+\langle 2Q(\phi_1z_1\beta_{k_1}+(\phi_2z_2+\bar{\phi}_2\bar{z}_2)\beta_{k_2},y)+Q(y,y),\beta_{k_1}\rangle\\
&+\langle D_1(\alpha)\Delta y(0),\beta_{k_1}\rangle ],
\end{array}
\end{equation}
\begin{equation}\label{eq6+1-1-1}
\begin{array}{rlc}
f^{12}_2(z, y,\alpha)=&\psi_2 (0)[L_1(\alpha)\phi_2z_2+L_1(\alpha)\bar{\phi}_2\bar{z}_2-\mu_{k_2}D_1(\alpha)(\phi_2(0)z_2+\bar{\phi}_2(0)\bar{z}_2)+\\

& Q_{\phi_1\phi_1}z_1^2\langle\beta_{k_1}^2,\beta_{k_2}\rangle+2(Q_{\phi_1\phi_2}z_1z_2+Q_{\phi_1\bar{\phi}_2}z_1\bar{z}_2)\langle\beta_{k_1}\beta_{k_2},\beta_{k_2}\rangle+(Q_{\phi_2\phi_2}z_2^2+\\

&2Q_{\phi_2\bar{\phi}_2}z_2\bar{z}_2+Q_{\bar{\phi}_2\bar{\phi}_2}\bar{z}_2^2 )\langle\beta_{k_2}^2,\beta_{k_2}\rangle+\langle L_1(\alpha)y,\beta_{k_2}\rangle+\langle 2Q(\phi_1z_1\beta_{k_1}+\\

&(\phi_2z_2+\bar{\phi}_2\bar{z}_2)\beta_{k_2},y)+Q(y,y),\beta_{k_2}\rangle+\langle D_1(\alpha)\Delta y(0),\beta_{k_2}\rangle ],
\end{array}
\end{equation}
\begin{equation}\label{eq6+2-1}
\begin{array}{rlc}
f^{1i}_3(z, 0,0)=&\psi_i (0)[C_{\phi_1\phi_1\phi_1}z_1^3\langle\beta_{k_1}^3,\beta_{k_i}\rangle+(C_{\phi_2\phi_2\phi_2}z_2^3+C_{\bar{\phi}_2\bar{\phi}_2\bar{\phi}_2}\bar{z}_2^3+3C_{\phi_2\phi_2\bar{\phi}_2}z_2^2\bar{z}_2+\\

&3C_{\phi_2\bar{\phi}_2\bar{\phi}_2}z_2\bar{z}_2^2)\langle\beta_{k_2}^3,\beta_{k_i}\rangle+3(C_{\phi_1\phi_1\phi_2}z_1^2z_2+C_{\phi_1\phi_1\bar{\phi}_2}z_1^2\bar{z}_2)\langle\beta_{k_1}^2\beta_{k_2},\beta_{k_i}\rangle+\\

&+3(C_{\phi_1\phi_2\phi_2}z_1z_2^2+C_{\phi_1\bar{\phi}_2\bar{\phi}_2}z_1\bar{z}_2^2+2C_{\phi_1\phi_2\bar{\phi}_2}z_1z_2\bar{z}_2)\langle\beta_{k_1}\beta_{k_2}^2,\beta_{k_i}\rangle],\ i=1,2,

\end{array}
\end{equation}
\begin{equation}\label{eq6+3-1}
\begin{array}{rlc}
f^{2}_2(z, 0,0)(\theta)=&\delta(\theta)[Q_{\phi_1\phi_1}z_1^2\beta_{k_1}^2+2(Q_{\phi_1\phi_2}z_1z_2+Q_{\phi_1\bar{\phi}_2}z_1\bar{z}_2)\beta_{k_1}\beta_{k_2}+\\

&+(Q_{\phi_2\phi_2}z_2^2+2Q_{\phi_2\bar{\phi}_2}z_2\bar{z}_2+
Q_{\bar{\phi}_2\bar{\phi}_2}\bar{z}_2^2 )\beta_{k_2}^2]-\{\phi_1(\theta)\psi_1 (0)[Q_{\phi_1\phi_1}z_1^2\langle\beta_{k_1}^2,\beta_{k_1}\rangle+\\

&+2(Q_{\phi_1\phi_2}z_1z_2+Q_{\phi_1\bar{\phi}_2}z_1\bar{z}_2)\langle\beta_{k_1}\beta_{k_2},\beta_{k_1}\rangle+(Q_{\phi_2\phi_2}z_2^2+2Q_{\phi_2\bar{\phi}_2}z_2\bar{z}_2+\\

&+Q_{\bar{\phi}_2\bar{\phi}_2}\bar{z}_2^2 )\langle\beta_{k_2}^2,\beta_{k_1}\rangle]\beta_{k_1}+(\phi_2(\theta)\psi_2 (0)+\bar{\phi}_2(\theta)\bar{\psi}_2 (0))[Q_{\phi_1\phi_1}z_1^2\langle\beta_{k_1}^2,\beta_{k_2}\rangle+

\\&+2(Q_{\phi_1\phi_2}z_1z_2+Q_{\phi_1\bar{\phi}_2}z_1\bar{z}_2)\langle\beta_{k_1}\beta_{k_2},\beta_{k_2}\rangle+(Q_{\phi_2\phi_2}z_2^2+2Q_{\phi_2\bar{\phi}_2}z_2\bar{z}_2+\\

&+Q_{\bar{\phi}_2\bar{\phi}_2}\bar{z}_2^2 )\langle\beta_{k_2}^2,\beta_{k_2}\rangle]\beta_{k_2}
\},~~~~~\text{for}~\theta\in[-r,0],
\end{array}
\end{equation}
where $\delta(\theta)=0$, for $\theta \in [-r,0)$,  $\delta(0)=1$.

Now we first obtain the normal form of \eqref{eq301} up to the quadratic terms.
\begin{lemma} \label{g2}
Assume that {\bf (H1)}-{\bf (H4)} are satisfied. Ignore the effect of the higher order terms $(\geq 2)$ of the perturbation parameter,  then the normal form   of \eqref{eq301} up to the quadratic terms on the center manifold at $\alpha=0$  has the form %(see \cite{Faria2})
\begin{equation}\label{eq440}
	\dot{z}=Bz+\frac{1}{2}g_2^1(z,0,\alpha)+h.o.t..
	\end{equation}
	Here % $g_2^1$ is the
	%function giving the quadratic terms in $(x,\mu)$ for $y=0$, and is
	%determined by
	%In fact,
    \begin{equation}\label{eq1}
	\begin{array}{rlc}
	g^{1}_2(z,0,\alpha)=&2a_{1}(\alpha)z_1+\psi_1 (0)[Q_{\phi_1\phi_1}z_1^2\langle\beta_{k_1}^2,\beta_{k_1}\rangle+2Q_{\phi_2\bar{\phi}_2}z_2\bar{z}_2\langle\beta_{k_2}^2,\beta_{k_1}\rangle]e_1\\
	
	&+2b_{2}(\alpha)z_2+\psi_2 (0)[2Q_{\phi_1\phi_2}z_1z_2\langle\beta_{k_1}\beta_{k_2},\beta_{k_2}\rangle]e_2\\
	
	&+\overline{b_{2}(\alpha)}z_2+\bar{\psi}_2 (0)[2Q_{\phi_1\bar{\phi}_2}z_1\bar{z}_2\langle\beta_{k_1}\beta_{k_2},\beta_{k_2}\rangle]e_3,
	
	\end{array}
	\end{equation}
	with 	
	\begin{equation}\label{a1b2}
	\begin{array}{rlc}
	a_{1}(\alpha)=&\frac{1}{2}\psi_1 (0)(L_1(\alpha)\phi_1-\mu_{k_1}D_1(\alpha)\phi_1(0)),\\
	b_{2}(\alpha)=&\frac{1}{2}\psi_2 (0)(L_1(\alpha)\phi_2-\mu_{k_2}D_1(\alpha)\phi_2(0)),
	\end{array}
	\end{equation}
	and $h.o.t.$ stands for higher order terms.
\end{lemma}
\begin{proof}
Let $M_2^1$ denote the operator
\begin{equation}\label{pppppp}
M_2^1: V_2^5(\mathbb{C}^3)\to
V_2^5(\mathbb{C}^3),~~\mathrm{and}~~(M_2^1p)(z,\alpha )=D_zp(z,\alpha
)Bz- Bp(z,\alpha ),
\end{equation}
where $V_2^5(\mathbb{C}^3)$ is the linear space of the second
order homogeneous polynomials in five variables
$(z_1,z_2,\bar{z}_2,\alpha_1,\alpha_2)$ with coefficients in
$\mathbb{C}^3$,$z=(z_1,z_2,\bar{z}_2),~\alpha=(\alpha_1,\alpha_2)$  and $B=diag(0,\mathrm{i}\omega_0, -\mathrm{i}\omega_0)$. One may choose the decomposition
$$V_2^5(\mathbb{C}^3)=\textrm{Im}(M_2^1)\oplus \textrm{Im}(M_2^1)^c$$
with complementary space $\textrm{Im}(M_2^1)^c$ spanned by the elements
\begin{equation}\label{eq6+6}
z_1^2e_1,~z_2\bar{z}_2e_1,~z_1\alpha_i e_1,~z_1z_2e_2,~z_2\alpha_i e_2,~z_1\bar{z}_2e_3,~\bar{z}_2\alpha_i e_3,\;\;i=1,2,
\end{equation}
where $e_1,e_2,e_3$ denote the natural basis of $\mathbb{R}^3$.
%$$
%\left(\begin{array}{cc} x_1^2\\ 0 \\0 \end{array} \right),
%\left(\begin{array}{cc}x_2x_3\\ 0\\0 \end{array} \right),
%\left(\begin{array}{cc} x_1\mu_i\\ 0\\0 \end{array} \right),
%\left(\begin{array}{cc} \mu_1^2\\ 0\\0 \end{array} \right),
%\left(\begin{array}{cc} \mu_2^2\\ 0\\0 \end{array} \right),
%\left(\begin{array}{cc} \mu_1\mu_2\\ 0\\0 \end{array} \right),
%$$
%$$
%\left(\begin{array}{cc} 0\\ x_1x_2\\ 0 \end{array} \right),
%\left(\begin{array}{cc} 0\\ x_2\mu_i\\ 0 \end{array} \right),
%\left(\begin{array}{cc} 0\\ 0\\ x_1x_3 \end{array} \right),
%\left(\begin{array}{cc} 0\\ 0\\ x_3\mu_i \end{array} \right),
% ~i=1,2 .
%$$
By the projection mapping which was presented in \cite{Faria2000},
\begin{equation}\label{eq4401}
g_2^1(z,0,\alpha)=Proj_{(\textrm{Im}(M_2^1))^c}f_2^1(z,0,\alpha),
\end{equation}
we get \eqref{eq1}, \eqref{a1b2} and that completes the proof.
\end{proof}

Let $V_2^{3}(\mathbb{C}^3\times \mathrm{Ker}\pi)$ be the space of
homogeneous polynomials of degree $2$ in the variables
$z=(z_1,z_2,z\bar{}_2)$ with coefficients in $\mathbb{C}^3\times
\mathrm{Ker} \pi.$
Let the operator $M_2^1$ defined in \eqref{pppppp} be restricted in $V_2^3(\mathbb{C}^3)$, as
\begin{equation}\label{eq9a}
M_2^1: V_2^3(\mathbb{C}^3)\longmapsto
V_2^3(\mathbb{C}^3),~~\mathrm{and}~~ (M_2^1p)(z)=D_zp(z
)Bz- Bp(z),
\end{equation}
and define the operator $M_2^2$ by
\begin{equation}\label{eq9}
M_2^2:~V_2^{3}(\mathcal{Q}^1)\subset V_2^3(\mathrm{Ker}\pi)~\longmapsto~V_2^{3}(\mathrm{Ker}\pi),~~\mathrm{and}~~
(M_2^2h)(z)=D_z h(z) B z - A_{1}(h(z)),
\end{equation}
% Define the operator $M_2=(M_2^1,M_2^2)$ by
% \begin{equation}\label{eq9}
% \begin{split}
% &M_2^1:V_2^{3}(\mathbb{C}^3)\to V_2^{3}(\mathbb{C}^3), \;\; \text{and } \; (M_2^1p)(z)= D_z p(z) B z-Bp(z),\\
% &M_2^2:V_2^{3}(\mathcal{Q}^1)\left(\subset V_2^3(\mathrm{Ker}\pi)\to V_2^{3}(\mathrm{Ker}\pi)\right),\; \text{and } (M_2^2h)(z)=D_z h(z) B z - A_{1}(h(z)).
% \end{split}
% \end{equation}
then we have the following decompositions:
\begin{equation} \label{eq10}
\begin{split}
V_2^{3}(\mathbb{C}^{3})&= \textrm{Im}(M_2^1)\oplus \textrm{Im}(M_2^1)^c, \;\;
V_2^{3}(\mathbb{C}^{3})= \mathrm{Ker}(M_2^1)\oplus \mathrm{Ker}(M_2^1)^c, \\
V_2^{3}(\mathrm{Ker} \pi)&=\mathrm{Im}(M_2^2)\oplus\mathrm{Im}(M_2^2)^c, \;\;
V_2^{3}(\mathcal{Q}^1) =\mathrm{Ker}(M_2^2)\oplus \mathrm{Ker}(M_2^2)^c.
\end{split}
\end{equation}
The projection associated with the preceding decomposition of
$V^{3}_2(\mathbb{C}^3)\times V^{3}_2(\mathrm{Ker}\pi)$ over
$\mathrm{Im}(M^1_2)\times \mathrm{Im}(M^2_2)$ is denoted by
$P_{I,2}=(P_{I,2}^1,P_{I,2}^2)$.

Following \cite{Faria2000}, we  set
\begin{equation}\label{eq14}
U_2(z)=\left(
	\begin{array}{c}
	U_2^1\\
	U_2^2
	\end{array}\right)
=M_2^{-1}P_{I,2}f_2(z,0,0),
\end{equation}
and by a
transformation of variables
\begin{equation}\label{eq17}
(z,y)=(\hat{z},\hat{y})+\frac{1}{2!}U_2(\hat{z}),
\end{equation}
the first  equation of \eqref{eq8} becomes, after dropping the hats,
\begin{equation}\label{eq18}
\dot{z}=Bz+\frac{1}{2!}g^1_2(z,0,\mu)+\frac{1}{3!}\overline{f}^1_{3}(z,0,0)+\cdots,
\end{equation}
where
\begin{equation}\label{eq19}
\overline{f}^1_3(z,0,0)=f^1_3(z,0,0)+\frac{3}{2}[(Df^1_2(z,y,0))_{y=0}U_2(z)-DU^1_2(z)g^1_2(z,0,0)].
\end{equation}

To complete the proof of Theorem \ref{g3}, we only need to calculate the third order term $g_3^1(z,0,0)$ in the normal form \eqref{eq441}. It is divided into three steps.
\\
%Therefor, we need deduce $U^1_2$ and $U^2_2$.\\
\noindent{\bf Step 1. Computation of $U_2^1$.}
\begin{lemma} Assume that {\bf (H1)}-{\bf (H4)} are satisfied. Then the  formula of  $U_2^1$ in \eqref{eq14} is
\begin{equation}\label{eq1+1}
	\begin{split}
	&i\omega_0U^{1}_2(z)\\
	=&\psi_1 (0)[2(Q_{\phi_1\phi_2}z_1z_2- Q_{\phi_1\bar{\phi}_2}z_1\bar{z}_2)\langle\beta_{k_1}\beta_{k_2},\beta_{k_1}\rangle
	+\frac{1}{2}(Q_{\phi_2\phi_2}z_2^2-Q_{\bar{\phi}_2\bar{\phi}_2}\bar{z}_2^2 )\langle\beta_{k_2}^2,\beta_{k_1}\rangle]e_1\\
	&+\psi_2 (0)[-Q_{\phi_1\phi_1}z_1^2\langle\beta_{k_1}^2,\beta_{k_2}\rangle-Q_{\phi_1\bar{\phi}_2}z_1\bar{z}_2\langle\beta_{k_1}\beta_{k_2},\beta_{k_2}\rangle
    +(Q_{\phi_2\phi_2}z_2^2-2Q_{\phi_2\bar{\phi}_2}z_2\bar{z}_2
	-\frac{1}{3}Q_{\bar{\phi}_2\bar{\phi}_2}\bar{z}_2^2 )\langle\beta_{k_2}^2,\beta_{k_2}\rangle]e_2\\
	&-\bar{\psi}_2 (0)[-Q_{\phi_1\phi_1}z_1^2\langle\beta_{k_1}^2,\beta_{k_2}\rangle-Q_{\phi_1\phi_2}z_1z_2\langle\beta_{k_1}\beta_{k_2},\beta_{k_2}\rangle
    +(Q_{\bar{\phi}_2\bar{\phi}_2}\bar{z}_2^2-2Q_{\phi_2\bar{\phi}_2}z_2\bar{z}_2
	-\frac{1}{3}Q_{\phi_2\phi_2}z_2^2 )\langle\beta_{k_2}^2,\beta_{k_2}\rangle]e_3.
	\end{split}
	\end{equation}	
\end{lemma}
\begin{proof}
	Since $U^1_2 \in \mathrm{Ker}(M_2^1)^c$, and $\mathrm{Ker}(M_2^1)^c$
	is spanned by
	\begin{equation}\label{eq6+5}
    \begin{split}
	&z_2^2e_1,~\bar{z}_2^2e_1,~z_1z_2e_1,~z_1\bar{z}_2e_1,~
	z_1^2e_2,~z_2^2e_2,~\bar{z}_2^2e_2,\\
    &z_1\bar{z}_2e_2,~z_2\bar{z}_2e_2,~
	z_1^2e_3,~z_2^2e_3,~\bar{z}_2^2e_3,~z_1z_2e_3,~z_2\bar{z}_2e_3,
    \end{split}
	\end{equation}
	The above elements are mapped by $M_2^1$ to, respectively,
	$$%\begin{equation}\label{eq6+4}
	\begin{array}{rrrrrrrr}
	 2i\omega_0z_2^2e_1,&-2i\omega_0\bar{z}_2^2e_1,&i\omega_0z_1z_2e_1,&-i\omega_0z_1\bar{z}_2e_1,&
	-i\omega_0z_1^2e_2,&i\omega_0z_2^2e_2,&-3i\omega_0\bar{z}_2^2e_2,&\\
	-2i\omega_0z_1\bar{z}_2e_2,&
	-i\omega_0z_2\bar{z}_2e_2,&
	 i\omega_0z_1^2e_3,&3i\omega_0z_2^2e_3,&-i\omega_0\bar{z}_2^2e_3,&2i\omega_0z_1z_2e_3,&i\omega_0z_2\bar{z}_2e_3.&
	\end{array}
	$$
	Then, by \eqref{eq14} and \eqref{eq6+1-1}, the expression \eqref{eq1+1} of $U_2^1$ is obtained and  the proof is completed.
\end{proof}
%it is straightforward to calculate
%{\small
%$$
%\left(\!\!\begin{array}{c}x_2^2\\0\\0
%\end{array}\!\!\right),
%\left(\!\!\begin{array}{c}x_3^2\\0\\0
%\end{array}\!\!\right),
%\left(\!\!\begin{array}{c}x_1x_2\\0\\0
%\end{array}\!\!\right),
%\left(\!\!\begin{array}{c}x_1x_3\\0\\0
%\end{array}\!\!\right),
%\left(\!\!\begin{array}{cc}0\\ x_1^2\\0
%\end{array}\!\!\right),
%\left(\!\!\begin{array}{cc}0\\ x_2^2\\0
%\end{array}\!\!\right),
%\left(\!\!\begin{array}{cc}0\\ x_3^2\\0
%\end{array}\!\!\right),
%\left(\!\!\begin{array}{cc}0\\ x_1x_3\\0
%\end{array}\!\!\right),
%\left(\!\!\begin{array}{cc}0\\ x_2x_3\\0
%\end{array}\!\!\right),$$$$
%\left(\!\!\begin{array}{cc}0\\ 0\\x_1^2
%\end{array}\!\!\right),
%\left(\!\!\begin{array}{cc}0\\0\\ x_2^2
%\end{array}\!\!\right),
%\left(\!\!\begin{array}{cc}0\\0\\ x_3^2
%\end{array}\!\!\right),
%\left(\!\!\begin{array}{cc}0\\0\\ x_1x_3
%\end{array}\!\!\right),
%\left(\!\!\begin{array}{cc}0\\0\\ x_2x_3
%\end{array}\!\!\right), \ (i=1,2).
%$$
%} Furthermore,

\noindent{\bf Step 2. Computation of $U^2_2$.}

We know that
\begin{equation}\label{eq15+1}
\frac{1}{2!}U_2^2\triangleq h(z)=(h^{(1)}(z),~h^{(2)}(z),\cdots,~h^{(m)}(z))^
{\mathrm{T}}\in V_2^3(\mathcal{Q}^1)\end{equation}
is the unique solution
of the equation
\begin{equation}\label{eq15}
(M_2^2h)(z)=\frac{1}{2!} f_2^2(z,0,0).
\end{equation}
Thus, by \eqref{eq9} and the definition of $A_1$, we have
\begin{equation}\label{eq16}
D_zh(z)Bz-\dot{h}(z)+X_0[\dot{h}(z)(0)-L_0h(z)-D_0\Delta h(z)(0)] =
\frac{1}{2!} f_2^2(z,0,0).
\end{equation}
where $\dot{h}$ denotes the derivative of $h(z)(\theta)$ respective to
$\theta$.
Expressing $h(z)$ in the general monomial form, we have
%$$h(z)=\sum_{|q|=2}h_q(z)^2,~q\in\mathbb{N}_0^3.$$
%namely, {\small
\begin{equation}\label{eq15+2}
\begin{split}
h(z)(\theta)=&h_{200}(\theta)z_1^2\!+h_{020}(\theta)z_2^2
+h_{002}(\theta)\bar{z}_2^2+h_{110}(\theta)z_1z_2\\
&+h_{101}(\theta)z_1\bar{z}_2+h_{011}(\theta)z_2\bar{z}_2,\;\;
\theta \in [-r,0].
\end{split}
\end{equation}
Hence \eqref{eq16} is equivalent to, for $\theta\in [-r,0]$,
\begin{equation}\label{eq16+}\begin{split}
&-\dot{h}_{200}(\theta)z_1^2-{\dot h}_{011}(\theta)z_2\bar{z}_2+[2i\omega_0h_{020}(\theta)-\dot{h}_{020}(\theta)]z_2^2+[-2i\omega_0h_{002}(\theta)-{\dot h}_{002}(\theta)]\bar{z}_2^2\\
&+[i\omega_0h_{110}(\theta)-{\dot h}_{110}(\theta)]z_1z_2+[-i\omega_0h_{101}(\theta)-{\dot h}_{101}(\theta)]z_1\bar{z}_2
= \frac{1}{2!} f_2^2(z,0,0)(\theta).
\end{split}
\end{equation}
For $\theta \in [-r,0)$,  comparing the coefficients of $z^q$ for $|q|=2$, $q\in \mathbb{N}_0^3$ in \eqref{eq6+3-1} and \eqref{eq16+}, we obtain that
\begin{equation}\label{eq16+1}
\begin{array}{rl}
\dot{h}_{200}(\theta) =&\frac{1}{2}[\phi_1(\theta)\psi_1 (0)Q_{\phi_1\phi_1}\langle\beta_{k_1}^2,\beta_{k_1}\rangle\beta_{k_1}+\\&
+(\phi_2(\theta)\psi_2 (0)+\bar{\phi}_2(\theta)\bar{\psi}_2 (0))Q_{\phi_1\phi_1}\langle\beta_{k_1}^2,\beta_{k_2}\rangle\beta_{k_2}],\\

\dot{h}_{011}(\theta) =&\phi_1(\theta)\psi_1 (0)Q_{\phi_2\bar{\phi}_2}\langle\beta_{k_2}^2,\beta_{k_1}\rangle\beta_{k_1}+\\&
+(\phi_2(\theta)\psi_2 (0)+\bar{\phi}_2(\theta)\bar{\psi}_2 (0))Q_{\phi_2\bar{\phi}_2}\langle\beta_{k_2}^2,\beta_{k_2}\rangle\beta_{k_2},\\

-2\mathrm{i}
\omega_0h_{020}(\theta)+\dot{h}_{020}(\theta) =&\frac{1}{2}[\phi_1(\theta)\psi_1 (0)Q_{\phi_2\phi_2}\langle\beta_{k_2}^2,\beta_{k_1}\rangle\beta_{k_1}+\\&
+(\phi_2(\theta)\psi_2 (0)+\bar{\phi}_2(\theta)\bar{\psi}_2 (0))Q_{\phi_2\phi_2}\langle\beta_{k_2}^2,\beta_{k_2}\rangle\beta_{k_2}],\\

-\mathrm{i}
\omega_0h_{110}(\theta)+\dot{h}_{110}(\theta) =&\phi_1(\theta)\psi_1 (0)Q_{\phi_1\phi_2}\langle\beta_{k_1}\beta_{k_2},\beta_{k_1}\rangle\beta_{k_1}+\\&
+(\phi_2(\theta)\psi_2 (0)+\bar{\phi}_2(\theta)\bar{\psi}_2 (0))Q_{\phi_1\phi_2}\langle\beta_{k_1}\beta_{k_2},\beta_{k_2}\rangle\beta_{k_2}.
\end{array}
\end{equation}
Solving these equations, we obtain that
\begin{equation}\label{eq16+2}
\begin{array}{rl}
h_{200}(\theta) =&h_{200}(0)+
\frac{1}{2}\{\theta\phi_1(0)\psi_1 (0)Q_{\phi_1\phi_1}\langle\beta_{k_1}^2,\beta_{k_1}\rangle\beta_{k_1}+\\&
+\frac{1}{i\omega_0}[(e^{i\omega_0\theta}-1)\phi_2(0)\psi_2 (0)-(e^{-i\omega_0\theta}-1)\bar{\phi}_2(0)\bar{\psi}_2 (0)]Q_{\phi_1\phi_1}\langle\beta_{k_1}^2,\beta_{k_2}\rangle\beta_{k_2}\},\\

h_{011}(\theta) =&h_{011}(0)+\theta\phi_1(0)\psi_1 (0)Q_{\phi_2\bar{\phi}_2}\langle\beta_{k_2}^2,\beta_{k_1}\rangle\beta_{k_1}+\\&
+\frac{1}{i\omega_0}[(e^{i\omega_0\theta}-1)\phi_2(0)\psi_2 (0)-(e^{-i\omega_0\theta}-1)\bar{\phi}_2(0)\bar{\psi}_2 (0)]Q_{\phi_2\bar{\phi}_2}\langle\beta_{k_2}^2,\beta_{k_2}\rangle\beta_{k_2},\\

h_{020}(\theta) =&h_{020}(0)e^{2i\omega_0\theta}+\frac{1}{2i\omega_0}\{ \frac{1}{2}(e^{2i\omega_0\theta}-1)\phi_1(0)\psi_1 (0)Q_{\phi_2\phi_2}\langle\beta_{k_2}^2,\beta_{k_1}\rangle\beta_{k_1}+\\&
+[(e^{2i\omega_0\theta}-e^{i\omega_0\theta})\phi_2(0)\psi_2 (0)+\frac{1}{3}(e^{2i\omega_0\theta}-e^{-i\omega_0\theta})\bar{\phi}_2(0)\bar{\psi}_2 (0)]Q_{\phi_2\phi_2}\langle\beta_{k_2}^2,\beta_{k_2}\rangle\beta_{k_2}\},\\

h_{110}(\theta) =&h_{110}(0)e^{i\omega_0\theta}+\frac{1}{i\omega_0} (e^{i\omega_0\theta}-1)\phi_1(0)\psi_1 (0)Q_{\phi_1\phi_2}\langle\beta_{k_1}\beta_{k_2},\beta_{k_1}\rangle\beta_{k_1}+\\&
+[\theta e^{i\omega_0\theta}\phi_2(0)\psi_2 (0)+\frac{1}{2i\omega_0}(e^{i\omega_0\theta}-e^{-i\omega_0\theta})\bar{\phi}_2(0)\bar{\psi}_2 (0)]Q_{\phi_1\phi_2}\langle\beta_{k_1}\beta_{k_2},\beta_{k_2}\rangle\beta_{k_2},\\
h_{002}(\theta) =&h_{020}(\theta),~~~~h_{101}(\theta) =h_{110}(\theta),~~~\theta \in [-r,0].
\end{array}
\end{equation}
And at $\theta=0$, by \eqref{eq16+2} we have
\begin{equation}\label{eq16+7}\begin{array}{rl}
&-(L_0(h_{200})+D_0\Delta h_{200}(0))z_1^2+[2i\omega_0h_{020}(0)-(L_0(h_{020})+D_0\Delta h_{020}(0))]z_2^2+\\&
[-2i\omega_0h_{002}(0)-(L_0(h_{002})+D_0\Delta h_{002}(0))]\bar{z}_2^2
+[i\omega_0h_{110}(0)-(L_0(h_{110})+D_0\Delta h_{110}(0))]z_1z_2\\&+[-i\omega_0h_{101}(0)-(L_0(h_{101})+D_0\Delta h_{101}(0))]z_1\bar{z}_2-(L_0(h_{011})+D_0\Delta h_{011}(0))z_2\bar{z}_2\\
=& \frac{1}{2!} f_2^2(z,0,0)(0),
\end{array}
\end{equation}
Again expanding the above sum and comparing the coefficients, we obtain that
\begin{equation}\label{eq16+3-}
\begin{array}{rl}
L_0(h_{200})+D_0\Delta h_{200}(0)=&\frac{1}{2}[-Q_{\phi_1\phi_1}\beta_{k_1}^2+\phi_1(0)\psi_1 (0)Q_{\phi_1\phi_1}\langle\beta_{k_1}^2,\beta_{k_1}\rangle\beta_{k_1}+\\&
(\phi_2(0)\psi_2 (0)+\bar{\phi}_2(0)\bar{\psi}_2 (0))Q_{\phi_1\phi_1}\langle\beta_{k_1}^2,\beta_{k_2}\rangle\beta_{k_2}],\\

L_0(h_{011})+D_0\Delta h_{011}(0) =&-Q_{\phi_2\bar{\phi}_2}\beta_{k_2}^2+\phi_1(0)\psi_1 (0)Q_{\phi_2\bar{\phi}_2}\langle\beta_{k_2}^2,\beta_{k_1}\rangle\beta_{k_1}+\\&
(\phi_2(0)\psi_2 (0)+\bar{\phi}_2(0)\bar{\psi}_2 (0))Q_{\phi_2\bar{\phi}_2}\langle\beta_{k_2}^2,\beta_{k_2}\rangle\beta_{k_2},\\

-2\mathrm{i}
\omega_0h_{020}(0)+L_0(h_{020})+D_0\Delta h_{020}(0) =&\frac{1}{2}[-Q_{\phi_2\phi_2}\beta_{k_2}^2+\phi_1(0)\psi_1 (0)Q_{\phi_2\phi_2}\langle\beta_{k_2}^2,\beta_{k_1}\rangle\beta_{k_1}+\\&
(\phi_2(0)\psi_2 (0)+\bar{\phi}_2(0)\bar{\psi}_2 (0))Q_{\phi_2\phi_2}\langle\beta_{k_2}^2,\beta_{k_2}\rangle\beta_{k_2}],\\

-\mathrm{i}
\omega_0h_{110}(0)+L_0(h_{110})+D_0\Delta h_{110}(0) =&-Q_{\phi_1\phi_2}\beta_{k_1}\beta_{k_2}+\phi_1(0)\psi_1 (0)Q_{\phi_1\phi_2}\langle\beta_{k_1}\beta_{k_2},\beta_{k_1}\rangle\beta_{k_1}+\\&
(\phi_2(0)\psi_2 (0)+\bar{\phi}_2(0)\bar{\psi}_2 (0))Q_{\phi_1\phi_2}\langle\beta_{k_1}\beta_{k_2},\beta_{k_2}\rangle\beta_{k_2}.
\end{array}
\end{equation}
Therefore $U_2^2$ are determined by \eqref{eq16+2} and \eqref{eq16+3-}.
For later computation of the third order normal form, note that for any $q\in \N_0^3$, $|q|=2$, we have
\begin{equation}\label{eq16+4}
h_q(\theta)=\langle h_q(\theta),\beta_{k_1}\rangle\beta_{k_1}+\langle h_q(\theta),\beta_{k_2}\rangle\beta_{k_2}+\sum_{k\geq 0,k\neq k_1,k_2}\langle h_q(\theta),\beta_{k}\rangle\beta_{k}.
\end{equation}
Then  $\langle h_q(0),\beta_{k}\rangle$ ($k\in \mathbb{N}_0$) can be obtained from \eqref{eq16+3-}, and $\langle h_q(\theta),\beta_{k}\rangle$ for $\theta\in [-r,0]$ is determined by \eqref{eq16+2}.
In fact, we do not need to find $\langle h_q(0),\beta_{k}\rangle$ for all $q\in\mathbb{N}_0^3$, $|q|=2$ and $k\in \mathbb{N}_0$, but only need to find the ones  appearing in $g_3^1(z,0,0)$.

\noindent{\bf Step 3. Computation of $g_3^1$.}

Now we have all the components for computing the third order normal form. Let $M_3$ be the operator defined in $V_3^3(\mathbb{C}^3\times Ker
\pi),$ with
$$
M_3^1: V_3^3(\mathbb{C}^3)\to
V_3^3(\mathbb{C}^3)~~\mathrm{and}~~ (M_3^1p)(z )=D_zp(z )Bz- Bp(z ),
$$
where $V_3^3(\mathbb{C}^3)$ denotes the linear space of  homogeneous polynomials of degree $3$ in the variables $z=(z_1,z_2,\bar{z}_2)$ with coefficients in $\mathbb{C}^3$. Then  one may choose the
decomposition
$$V_3^3(\mathbb{C}^3)=\text{Im}(M_3^1)\oplus \text{Im}(M_3^1)^c$$
with the complementary space $(\text{Im}(M_3^1))^c$ spanned by the elements
\begin{equation}\label{eq6+4}
z_1^3e_1,~z_1z_2\bar{z}_2e_1,~z_1^2z_2e_2,~z_2^2\bar{z}_2e_2,~z_1^2\bar{z}_2e_3,~z_2\bar{z}_2^2 e_3,
\end{equation}
where $e_1,e_2,e_3$ denote the natural basis of $\mathbb{R}^3$.
Now we have the normal form up to the third order
\begin{equation}\label{eq44}
\dot{z}=Bz+\dfrac{1}{2!}g^{1}_2(z,0,\alpha)+\dfrac{1}{3!}g^{1}_3(z,0,0)+h.o.t.,
\end{equation}
where \ba \label{eq444}
\dfrac{1}{3!}g^1_3(z,0,0)=\dfrac{1}{3!}Proj_{(\textrm{Im}(M_3^1))^c}\overline{f}^1_3(z,0,0).
\ea
From \eqref{eq19} denoting
\begin{equation}\label{eq44+}
\begin{array}{rlc}
g_{31}(z)=&\dfrac{1}{6} Proj_{(\textrm{Im}(M_3^1))^c}f^1_3(z,0,0),\\
g_{32}(z)=&-\dfrac{1}{4} Proj_{(\textrm{Im}(M_3^1))^c}D_zU^1_2(z)g^1_2(z,0,0),\\
g_{33}(z)=&\dfrac{1}{4} Proj_{(\textrm{Im}(M_3^1))^c}(D_zf^1_2(z,y,0))_{y=0}U_2^1(z),\\
g_{34}(z)=&\dfrac{1}{4} Proj_{(\textrm{Im}(M_3^1))^c}(D_yf^1_2(z,y,0))_{y=0}U_2^2(z),
\end{array}
\end{equation}
then \ba \label{eq44+1}
\dfrac{1}{3!}g^1_3(z,0,0)=g_{31}(z)+g_{32}(z)+g_{33}(z)+g_{34}(z).
\ea
From \eqref{eq6+2-1} and \eqref{eq44+}, we  obtain
	\begin{equation}\label{g31}
	\begin{array}{rlc}
	g_{31}(z)=&~\frac{1}{6}\psi_1 (0)[C_{\phi_1\phi_1\phi_1}z_1^3\langle\beta_{k_1}^3,\beta_{k_1}\rangle +6 C_{\phi_1\phi_2\bar{\phi}_2}z_1z_2\bar{z}_2\langle\beta_{k_1}\beta_{k_2}^2,\beta_{k_1}\rangle ]e_1\\	
	
	&+\frac{1}{2}\psi_2 (0)[C_{\phi_2\phi_2\bar{\phi}_2}z_2^2\bar{z}_2\langle\beta_{k_2}^3 ,\beta_{k_2}\rangle+C_{\phi_1\phi_1\phi_2}z_1^2z_2\langle\beta_{k_1}^2\beta_{k_2},\beta_{k_2}\rangle ]e_2\\
	
	&+\frac{1}{2}\bar{\psi}_2 (0)[C_{\phi_2\bar{\phi}_2\bar{\phi}_2}z_2\bar{z}_2^2\langle\beta_{k_2}^3 ,\beta_{k_2}\rangle+C_{\phi_1\phi_1\bar{\phi}_2}z_1^2\bar{z}_2\langle\beta_{k_1}^2\beta_{k_2},\beta_{k_2}\rangle ]e_3.
\end{array}
\end{equation}

For $g_{32}$, since $U^1_2(z)\in \text{Ker}(M_2^1)^c$ and $g^1_2(z,0,0)\in \text{Im}(M_2^1)^c$, by \eqref{eq6+5} and \eqref{eq6+6} we set
	$$%\begin{equation}\label{eq1+1}
	\begin{array}{rlc}
	U^{1}_2(z)=&
	(a_{020}^{(1)}z_2^2+a_{002}^{(1)}\bar{z}_2^2 +a_{110}^{(1)}z_1z_2 +a_{101}^{(1)}z_1\bar{z}_2)e_1+\\
	
	&+(a_{200}^{(2)}z_1^2+a_{020}^{(2)}z_2^2+a_{002}^{(2)}\bar{z}_2^2 +a_{101}^{(2)}z_1\bar{z}_2 +a_{011}^{(2)}z_2\bar{z}_2)e_2+\\
	&+(a_{200}^{(3)}z_1^2+a_{020}^{(3)}z_2^2+a_{002}^{(3)}\bar{z}_2^2 +a_{110}^{(3)}z_1z_2 +a_{011}^{(3)}z_2\bar{z}_2 )e_3,
	\end{array}
	$$%\end{equation}
	$$
	 g^{1}_2(z,0,0)=(b_{200}^{(1)}z_1^2+b_{011}^{(1)}z_2\bar{z}_2)e_1+b_{110}^{(2)}z_1z_2e_2+b_{101}^{(3)}z_1\bar{z}_2e_3,
	$$
	then
	$$\begin{array}{rlc}
	D_zU^1_2(z)g^1_2(z,0,0)=&[(a_{110}^{(1)}z_2 +a_{101}^{(1)}\bar{z}_2)(b_{200}^{(1)}z_1^2+b_{011}^{(1)}z_2\bar{z}_2)+(2a_{020}^{(1)}z_2+a_{110}^{(1)}z_1)b_{110}^{(2)}z_1z_2\\&
	 +(2a_{002}^{(1)}\bar{z}_2+a_{101}^{(1)}z_1)b_{101}^{(3)}z_1\bar{z}_2]e_1+[(2a_{200}^{(2)}z_1+a_{101}^{(2)}\bar{z}_2)(b_{200}^{(1)}z_1^2+b_{011}^{(1)}z_2\bar{z}_2)\\&
	+(2a_{020}^{(2)}z_2+a_{011}^{(2)}\bar{z}_2)b_{110}^{(2)}z_1z_2
	 +(2a_{020}^{(2)}\bar{z}_2+a_{101}^{(2)}z_1+a_{011}^{(2)}z_2)b_{101}^{(3)}z_1\bar{z}_2]e_2\\&
	 +[(2a_{200}^{(3)}z_1+a_{110}^{(3)}z_2)(b_{200}^{(1)}z_1^2+b_{011}^{(1)}z_2\bar{z}_2)
	+(2a_{020}^{(3)}z_2+a_{110}^{(3)}z_1\\&
	 +a_{011}^{(3)}\bar{z}_2)b_{110}^{(2)}z_1z_2+(2a_{020}^{(3)}\bar{z}_2+a_{011}^{(3)}z_2)b_{101}^{(3)}z_1\bar{z}_2]e_3\\
	\in & \textrm{Im}(M_3^1).
	\end{array}$$
	This implies that
\begin{equation}\label{g32}
	g_{32}(z)=-\dfrac{1}{4} Proj_{(I_m(M_3^1))^c}DU^1_2(z)g^1_2(z,0,0)=0.
\end{equation}
By using \eqref{eq44+}, \eqref{eq6+1}, \eqref{eq1+1} and \eqref{eq6+4}, we obtain
	\begin{equation}\label{g33}
	 g_{33}(z)=g_{33}^{(1)}(z)e_1+g_{33}^{(2)}(z)e_2+\overline{g_{33}^{(2)}(z)}e_3,\end{equation}
	with
	\begin{equation}\label{eq6+4-1}
	\begin{array}{rlc}
	g_{33}^{(1)}(z)=&\frac{1}{2i\omega_0}\psi_1 (0)[-Q_{\phi_1\phi_2}\psi_2 (0)+Q_{\phi_1\bar{\phi}_2}\bar{\psi}_2 (0)]Q_{\phi_1\phi_1}\langle\beta_{k_1}\beta_{k_2},\beta_{k_1}\rangle\langle\beta_{k_1}^2,\beta_{k_2}\rangle z_1^3+\\&
	
	+\frac{1}{i\omega_0}\psi_1 (0)\{[-Q_{\phi_1\phi_2}\psi_2 (0)+Q_{\phi_1\bar{\phi}_2}\bar{\psi}_2 (0)]Q_{\phi_2\bar{\phi}_2}\langle\beta_{k_1}\beta_{k_2},\beta_{k_1}\rangle\langle\beta_{k_2}^2,\beta_{k_2}\rangle +\\&
	
	+\frac{1}{2}[-Q_{\phi_2\phi_2}\psi_2 (0)Q_{\phi_1\bar{\phi}_2}+Q_{\bar{\phi}_2\bar{\phi}_2}\bar{\psi}_2 (0)Q_{\phi_1\phi_2}]\langle\beta_{k_1}\beta_{k_2},\beta_{k_2}\rangle\langle\beta_{k_2}^2,\beta_{k_1}\rangle]\}z_1z_2\bar{z}_2,
	\end{array}
	\end{equation}
	\begin{equation}\label{eq6+4-2}
	\begin{array}{rlc}
	g_{33}^{(2)}(z)=&\frac{1}{2i\omega_0}\psi_2 (0)\{[2Q_{\phi_1\phi_1}\psi_1 (0)\langle\beta_{k_1}\beta_{k_2},\beta_{k_1}\rangle\langle\beta_{k_1}^2,\beta_{k_2}\rangle +Q_{\phi_1\bar{\phi}_2}\bar{\psi}_2 (0)\langle\beta_{k_1}\beta_{k_2},\beta_{k_2}\rangle^2]Q_{\phi_1\phi_2} +\\&
	
	[-Q_{\phi_2\phi_2}\psi_2 (0)+Q_{\phi_2\bar{\phi}_2}\bar{\psi}_2 (0)]Q_{\phi_1\phi_1}\langle\beta_{k_1}^2,\beta_{k_2}\rangle\langle\beta_{k_2}^2,\beta_{k_2}\rangle \}z_1^2z_2+\\&
	
	\frac{1}{4i\omega_0}\psi_2 (0)\{Q_{\phi_1\bar{\phi}_2}\psi_1 (0)Q_{\phi_2\phi_2}\langle\beta_{k_1}\beta_{k_2},\beta_{k_2}\rangle\langle\beta_{k_2}^2,\beta_{k_1}\rangle +
	\frac{2}{3}Q_{\bar{\phi}_2\bar{\phi}_2}\bar{\psi}_2 (0)Q_{\phi_2\phi_2}\langle\beta_{k_2}^2,\beta_{k_2}\rangle^2+\\&
	
	+[-2Q_{\phi_2\phi_2}\psi_2 (0)+4Q_{\phi_2\bar{\phi}_2}\bar{\psi}_2 (0)]Q_{\phi_2\bar{\phi}_2}\langle\beta_{k_2}^2,\beta_{k_2}\rangle^2\}z_2^2\bar{z}_2.
	\end{array}
	\end{equation}
Finally from \eqref{eq15+2} and the symmetric multilinearity of $Q$, we obtain
\begin{equation}\label{eq15+3}
		Q_{\phi h}Q_{\phi h_{200}}z_1^2+Q_{\phi h_{020}}z_2^2
		+Q_{\phi h_{002}}\bar{z}_2^2
		+Q_{\phi h_{110}}z_1z_2+Q_{\phi h_{101}}z_1\bar{z}_2
		+Q_{\phi h_{011}}z_2\bar{z}_2,
\end{equation}
with $\phi\in\{\phi_1,\phi_2,\bar{\phi}_2\}$. From  \eqref{eq44+},  \eqref{eq6+1-1}, \eqref{eq15+1} and \eqref{eq6+4}, we obtain
	\begin{equation}\label{g34}
	\begin{array}{rlc}
	g_{34}(z)=&\frac{1}{2}\psi_1 (0)[\langle Q_{\phi_1 h_{200}}\beta_{k_1},\beta_{k_1}\rangle z_1^3+(\langle Q_{\phi_1 h_{011}}\beta_{k_1},\beta_{k_1}\rangle +
	\langle Q_{\phi_2 h_{101}}\beta_{k_2},\beta_{k_1}\rangle+\\&
	\langle Q_{\bar{\phi}_2 h_{110}}\beta_{k_2},\beta_{k_1}\rangle)z_1z_2\bar{z}_2]e_1+\frac{1}{2}\psi_2 (0)[(\langle Q_{\phi_1 h_{110}}\beta_{k_1},\beta_{k_2}\rangle+
	\langle Q_{\phi_2 h_{200}}\beta_{k_2},\beta_{k_2}\rangle)z_1^2z_2+\\&(\langle Q_{\phi_2 h_{011}}\beta_{k_2},\beta_{k_2}\rangle+\langle Q_{\bar{\phi}_2 h_{020}}\beta_{k_2},\beta_{k_2}\rangle)z_2^2\bar{z}_2]e_2+\frac{1}{2}\bar{\psi}_2 (0)[(\langle Q_{\phi_1 h_{101}}\beta_{k_1},\beta_{k_2}\rangle+\\&
	\langle Q_{\bar{\phi}_2 h_{200}}\beta_{k_2},\beta_{k_2}\rangle)z_1^2\bar{z}_2+(\langle Q_{\bar{\phi}_2 h_{011}}\beta_{k_2},\beta_{k_2}\rangle+\langle Q_{\phi_2 h_{002}}\beta_{k_2},\beta_{k_2}\rangle)z_2\bar{z}_2^2]e_3.
	\end{array}
	\end{equation}

Now we can complete the proof of Theorem \ref{g3}.
\begin{proof}[Proof of Theorem \ref{g3}]
The conclusion of Theorem \ref{g3} follows from Lemma \ref{g2}, \eqref{g31}, \eqref{g32}, \eqref{g33} and \eqref{g34}.
\end{proof}

\section{Conclusion}

In this paper the normal forms up to the third order for a Hopf-steady state bifurcation of a general system of partial functional differential equations (PFDEs) is derived based on the center manifold and normal form theory of PFDEs. This is a codimension-two degenerate bifurcation with the characteristic equation having a pair of simple purely imaginary roots and a simple zero root, and the corresponding eigenfunctions may be spatially inhomogeneous. The PFDEs are reduced to a three-dimensional system of ordinary differential equations and precise dynamics near bifurcation point can be revealed by two unfolding parameters which can be expressed by those original perturbation parameters. Usually, the third order normal form is sufficient for analyzing bifurcation phenomena in most of the applications.

The normal forms for the Hopf-steady state bifurcation in a general PFDE has been investigated within the framework of Faria \cite{Faria2000,Faria2002}. In \cite{Faria2000} an important conclusion is that the normal forms up to  a certain finite order for both the PFDEs and its associated FDEs are the same under the assumption ($H5$).
And when ($H5$) is not satisfied, the associated FDEs may not provide complete information, and  further general results on the normal forms are not given in \cite{Faria2000}. In fact, the assumption ($H5$) is not satisfied when a Hopf-steady state bifurcation occurs. Our results on computing the normal forms on center manifolds, that is \eqref{third}-\eqref{eq03-5+++++}, \eqref{eq16+2} and \eqref{eq16+3-}, do not require ($H5$), which makes the approach applicable to a wider class of systems.

For more concrete expressions, we provide explicit formulas of the coefficients in the third order  normal forms in the Hopf-steady state bifurcation for delayed reaction-diffusion equations with Neumann boundary condition, that is \eqref{eq03-1}-\eqref{eq03-5}, and this includes the important case of  Turing-Hopf bifurcation (with $k_1\neq 0$). The formulas are user-friendly as they are expressed directly by the Fr\`echet derivatives of the functions up to third orders and the characteristic functions of the original systems, and they are shown in concise  matrix form which is convenient for computer implementation. These results can also be applied to reaction diffusion equations without delay, for example, see \cite{CJ2018}, and delayed functional differential equations without diffusion, see \cite{JiangW2010}.

Our general results are applied to the diffusive Schnakenberg system of biochemical reactions with gene expression time delay to demonstrate  how  our formulas can be applied in practical examples. In particular we provide specific conditions on the parameters for the existence and stability of spatially nonhomogeneous steady state solutions and time-periodic solutions near the Turing-Hopf bifurcation point. Our specific examples are for one-dimensional spatial domain with Neumann boundary conditions, but our  general framework is broad enough for high-dimensional spatial domains and Dirichlet boundary conditions. More specific computations for these cases will be done in the future.

% \section*{Acknowledgments}

	%%%%%%%%%%%%%%%%%%%%%%%%%%%%%%%%%%%%%%%%%
	%%%%%%
	%%%%%%
	%%%%%%%%%%%%%%%%%%%%%%%%%%%%%%%%%%%%%%%%%

	%\bibliographystyle{plainnab}
	%\bibliography{mybibfile.bib}
\bibliographystyle{plain}
\bibliography{mybibfile}

\end{document}